\documentclass[10pt, letterpaper,reqno]{amsart}
\usepackage{amssymb}
\usepackage{amsmath}
\usepackage{amsfonts, color}
\usepackage{graphicx}
\usepackage{mathrsfs,amssymb, slashed, cite}

\usepackage{hyperref}
\usepackage{cleveref}

\newcommand{\R}{\mathbb{R}}
\newcommand{\C}{\mathbb{C}}

\newcommand{\Z}{\mathbb{Z}}

\newtheorem{corollary}{Corollary}[section]

\newtheorem{lemma}{Lemma}[section]

\newtheorem{proposition}{Proposition}[section]

\newtheorem{theorem}{Theorem}[section]
\numberwithin{equation}{section}

\newcommand{\eps}{\varepsilon}

\newcommand{\qtq}[1]{\quad\text{#1}\quad}

\begin{document}

\title[Scattering for NLS with inhomogeneous nonlinearities]{Scattering for the $1d$ NLS with \\ inhomogeneous nonlinearities.}

\author[L. Baker]{Luke Baker}
\email{lukebake@uoregon.edu}
\address{Department of Mathematics, University of Oregon, Eugene, OR, USA}

\author[J. Murphy]{Jason Murphy}
\email{jamu@uoregon.edu}
\address{Department of Mathematics, University of Oregon, Eugene, OR, USA}

\maketitle

\begin{abstract} We prove large-data scattering in $H^1$ for inhomogeneous nonlinear Schr\"odinger equations in one space dimension for powers $p>2$. We assume the inhomogeneity is nonnegative and repulsive; we additionally require decay at infinity in the case $2<p\leq 4$.  We use the method of concentration-compactness and contradiction, utilizing a Morawetz estimate in the style of Nakanishi in order to preclude the existence of compact solutions. 
\end{abstract}

%\tableofcontents

%%%%%%%%%%%%%%%%%%%%%%
%%%%%%%%%%%%%%%%%%%%%%
\section{Introduction}

Our interest in this work is the large-data scattering theory for inhomogeneous nonlinear Schr\"odinger equations of the form
\begin{equation}\label{nls}
(i\partial_t + \Delta) u = a(x) |u|^p u.
\end{equation}
Here $u:\R_t\times\R_x^d\to\C$ and $a:\R^d\to\C$. The case $a\equiv 1$ corresponds to the standard defocusing power-type NLS, an extensively-studied model that we will review below. Throughout this paper, we will assume that the \emph{inhomogeneity} $a$ is nonnegative, so that the nonlinearity remains defocusing. The other essential assumption we will take is that of \emph{repulsivity}, which refers to the sign condition $x\cdot\nabla a\leq 0$.  These two conditions allow for the development of Morawetz estimates that play a central role in the large-data scattering theory.

For reasons to be explained shortly, our main result addresses only the one-dimensional case.  Our main result is the following: 

\begin{theorem}\label{T} Let $a:\R\to\R$ satisfy $a\geq 0$ and $x\cdot\nabla a\leq 0$.  Let $p>2$, and suppose 
\[
\begin{cases}
a,\nabla a \in L^\infty & \text{if }p>4 \\
a,\nabla a \in L^1\cap L^\infty & \text{if }2< p\leq 4.
\end{cases}
\]
Suppose further that $|\nabla a(x)|\to 0$ as $|x|\to\infty$.

For any $u_0\in H^1(\R)$, there exists a unique global solution $u:\R\times\R\to\C$ to \eqref{nls} with $u|_{t=0}=u_0$.  Moreover, $u$ scatters in $H^1$ as $t\to\pm\infty$; that is, there exist $u_\pm\in H^1(\R)$ such that
\[
\lim_{t\to\pm\infty} \|u(t)-e^{it\Delta}u_\pm\|_{H^1} = 0,
\]
where $e^{it\Delta}$ is the free Schr\"odinger group.  %Moreover $u,\nabla u\in L_t^\infty L_x^2\cap L_t^4 L_x^\infty(\R\times\R)$. 
\end{theorem}

To properly motivate Theorem~\ref{T}, let us first briefly review the large-data scattering theory for \eqref{nls} with $a\equiv 1$.  An $H^1$ scattering theory is available in the regime $\tfrac{4}{d}\leq p\leq \tfrac{4}{d-2}$ (with the upper bound always replaced by $p<\infty$ for $d\in\{1,2\}$).  The cases $p=\tfrac{4}{d}$ and $p=\tfrac{4}{d-2}$ are known as the \emph{mass-critical} and \emph{energy-critical} cases, respectively, and they require considerably more effort than the \emph{intercritical} regime.  As we will see, these cases are excluded from our main result, and therefore we will focus our discussion on the intercritical regime in what follows. 

In dimensions $d\geq 3$, $H^1$ scattering in the intercritical regime can be proven using a straightforward application of Morawetz estimates.  Broadly speaking, Morawetz estimates for \emph{defocusing} NLS come in two varieties: Lin--Strauss Morawetz type estimates (also known as the radial Morawetz estimates) and interaction Morawetz estimates\footnote{The closely-related virial identity is used more often in focusing problems, typically in the context of the concentration-compactness approach.}  Both types of estimates can be used to obtain scattering (even for non-radial data), although it turns out to be much simpler to use interaction Morawetz estimates.  We refer the reader to \cite{S, C, M-expo} for more details. 

In dimensions $d\in\{1,2\}$, the standard Morawetz weights no longer lead to coercive estimates.  Nonetheless, Nakanishi succeeded in developing new Lin--Strauss type Morawetz estimates in dimensions $d\in\{1,2\}$, which he used to obtain scattering in the intercritical regime (see \cite{N}).  As these estimates do not directly imply critical spacetime bounds, Nakanishi's proof of scattering relied on an implementation of the induction on energy argument pioneered by Bourgain \cite{Bourgain}. This problem was later revisited in \cite{CHVZ, CGT}, and the proof of scattering was greatly simplified by means of new interaction Morawetz estimates established in dimensions $d\in\{1,2\}$. 

We now return to the discussion of inhomogeneous NLS.  The model that has received the most attention to date involves nonlinearities of the form $\pm|x|^{-b}|u|^p u$ with $b>0$, and indeed our treatment of \eqref{nls} shares many similarities with the analysis of this special case.  Particular attention has been given to the focusing version of this equation and the problem of scattering below the ground state, although the defocusing problem has been considered as well (see e.g. \cite{MMZ, CFGM, M-INLS, FarahG, CC, CLZ, CLZ2, GM, AGT, LMZ}). Because this model retains a scaling symmetry, the virial identity continues to be a particularly effective tool in the analysis (as in the case of the standard power-type NLS).

The inclusion of a more general inhomogeneity generally destroys the scaling symmetry of the equation.  Nonetheless, the inclusion of a nonnegative repulsive inhomogeneity turns out to be compatible with the calculations used to obtain Lin--Strauss Morawetz type estimates.  As in the case $a\equiv 1$, it is then fairly straightforward to use such estimates to prove $H^1$ scattering in the intercritical regime in dimensions $d\geq 3$ (see e.g. \cite{W}). However, because the inhomogeneity breaks translation invariance, it is not clear whether one can hope to prove \emph{interaction} type Morawetz estimates in this setting.  

This leaves the question of scattering in the presence of repulsive inhomogeneities in dimensions $d\in\{1,2\}$.  In this setting, it turns out that one can still prove Morawetz estimates of the type appearing in \cite{N}; however, as described above, these estimates do not immediately imply scattering via control of some critical spacetime norm.  As in the case of \cite{N}, one therefore needs to utilize some kind of induction on energy argument.  In this work, we employ the closely-related strategy of concentration-compactness and contradiction in order to prove scattering.

In Theorem~\ref{T}, the cases $p>4$ correspond to the standard intercritical regime for NLS.  We note that the assumptions on $a$ in $p>4$ guarantee that $a(x)\to a_\pm$ as $x\to\pm\infty$ for some $a_\pm\geq 0$.  In particular, the inhomogeneity is not necessarily assumed to be localized in space (although it must be asymptotically constant as $x\pm\infty$).\footnote{We remark that the assumption $|\nabla a(x)|\to 0$ as $|x|\to\infty$ is included for the sake of some technical simplification in Proposition~\ref{P:SFAP} below, particularly in the range $p>4$. In fact, this assumption could be removed, for example by developing a stability theory using smallness in norms with only fractional derivatives.} In this work, we also observe that by imposing some decay on the inhomogeneity, we can actually lower the power of the nonlinearity while still treating the problem as `intercritical'.  One way to see this is to consider the following nonlinear estimate arising via $1d$ Strichartz estimates in the Duhamel formulation of the problem: 
\[
\| a |u|^p u \|_{L_t^{\frac43} L_x^1} \lesssim \|a\|_{L_x^\rho} \|u\|_{L_t^{2p} L_x^r}^{p} \|u\|_{L_t^4 L_x^\infty},\quad r=\tfrac{p\rho}{\rho-1}.
\]
For NLS problems, the space $L_t^{2p} L_x^r$ scales like $\dot H^s$, where $s=\tfrac12-\tfrac{2}{p}+\tfrac{1}{p\rho}$. In particular, when $\rho=\infty$ the value $s$ corresponds to the usual critical regularity of the problem with $a\equiv 1$. However, by choosing $\rho$ finite, we can increase the regularity of this space.  Under the strongest assumption $\rho=1$, the regularity becomes $s=\tfrac12-\tfrac1p$.  Thus the problem remains `intercritical' in some sense provided $p>2$.  As we will see, many of our techniques break down precisely as we approach the $p=2$ endpoint.  In fact, the case $p=2$ with $a\in L^1$ is a particularly interesting case, as it now resembles a mass-critical problem (cf. \cite{HKV}).  We plan to address this problem in future work. 

Similar considerations in dimensions $d\geq 2$ suggest that an $H^1$ scattering theory should be possible for arbitrarily small $p>0$.  In fact, it is straightforward to prove this in dimensions $d\geq 3$ by a fairly direct implementation of Lin--Strauss type Morawetz estimates (and by taking slightly stronger assumptions on the decay of $a$).  We contend that the same result will hold in dimension $d=2$.  However, as low-power nonlinearities introduce various new technical issues, we have elected to focus on the case $d=1$ in this paper, and plan to address the case $d=2$ in subsequent work. 

To close this discussion, let us also remark that a scattering theory for the standard power-type NLS is possible even for mass-subcritical nonlinearities; however, such results typically require weighted assumptions on the initial data (see e.g. \cite{CW, BGTV, TY, M-expo}).  In particular, for $\tfrac{2}{d}<p<\tfrac{4}{d}$, initial data in $\Sigma=\{f\in H^1:xf\in L^2\}$ leads to scattering in $H^1$; in fact, the scattering can be shown to hold in $\Sigma$ provided $p$ exceeds the `Strauss exponent'.  The essential `Morawetz-type' estimate used in this context is the pseudoconformal energy estimate, which is based on the virial identity.  As a version of this estimate can still be shown to hold in the presence of a nonnegative repulsive inhomogeneity, we can also adapt this approach in a straightforward way to prove $H^1$ scattering for \eqref{nls} with initial data in $\Sigma$.  In order to work with $H^1$ data as in Theorem~\ref{T}, however, it seems to be necessary to utilize some kind of induction on energy argument.  

In the rest of the introduction, we briefly outline the proof of Theorem~\ref{T}.  

First, under the assumptions of Theorem~\ref{T}, it is not difficult to obtain global well-posedness and $H^1$ boundedness for \eqref{nls}.  One can first obtain subcritical-type local well-posedness in $H^1$ via the usual arguments (i.e. one constructs a solution to the Duhamel formula using Strichartz estimates and a contraction mapping argument). To obtain global well-posedness, one observes that solutions to \eqref{nls} enjoy conservation of mass and energy, where the mass and energy are defined by 
\begin{equation}\label{MEdef}
M(u) = \int_\R |u(x)|^2\,dx \qtq{and} E(u) = \int_\R \tfrac12 |\nabla u(x)|^2 + \tfrac{a(x)}{p+2}|u(x)|^{p+2}\,dx,
\end{equation}
respectively. Using these conservation laws and nonnegativity of $a$, one can readily obtain \emph{a priori} $H^1$ bounds for solutions.  Under the assumptions of Theorem~\ref{T}, such bounds readily imply global existence for solutions to \eqref{nls}.  For more details, we refer the reader to Section~\ref{S:LWP} (and more generally, to the textbook \cite{C}). 

As mentioned above, although it is possible to establish $1d$ Morawetz estimates in the presence of a nonnegative repulsive nonlinearity, these estimates do not immediately imply sufficient control over space-time norms to obtain scattering.  In this work, we utilize the concentration-compactness approach to bring these estimates to bear.  The argument proceeds as follows:  

First, the arguments used to obtain local well-posedness show that scattering holds for solutions with sufficiently small data in $H^1$.  In particular, scattering holds if $M(u)+E(u)$ is sufficiently small.  Thus if scattering fails in general, there must be some critical value of $M(u)+E(u)$, say $E_c$, at which we can witness this failure, in the sense that we can construct a (necessarily compact) solution with $M(u)+E(u)=E_c$ that fails to scatter.  The strategy for constructing such a solution (using linear and nonlinear profile decompositions) is standard by now, but the presence of the inhomogeneity introduces some new challenges in our setting, particularly due to the broken translation symmetry\footnote{We note that the broken scaling symmetry does not present any significant difficulty in our setting, as our problem is scaling subcritical.}.  The most important technical point is the construction of scattering solutions corresponding to initial data translated far from the origin (see Proposition~\ref{P:SFAP}).  Drawing inspiration from works such as \cite{KMVZZ, KMVZ, MMZ, CFGM} (including works on the inhomogeneous NLS with nonlinearities of the form $|x|^{-b}|u|^p u$), we accomplish this by using scattering solutions to the appropriate limiting model (which is either the linear Schr\"odinger equation or the standard power-type NLS, depending on the limit of $a(x)$ at $\pm\infty$) and appealing to the stability theory for \eqref{nls}.  In particular, this part of the argument relies crucially on the existing scattering theory for the pure power-type NLS in the range $p>4$.  We ultimately prove that if Theorem~\ref{T} fails, then we can construct a solution to \eqref{nls} that has pre-compact orbit in $H^1$.  Importantly, because of Proposition~\ref{P:SFAP}, there is no need to include a moving spatial in order to obtain compactness. For complete details, see Section~\ref{S:Reduction}. 

In Section~\ref{S:Preclusion}, we prove Morawetz estimates for \eqref{nls} (in the spirit of Nakanishi \cite{N}).  We then demonstrate that these estimates are incompatible with the compact solutions constructed in Section~\ref{S:Reduction}, thus completing the proof of Theorem~\ref{T}.  As mentioned above, the estimates we prove are \emph{not} interaction-type Morawetz estimates, and indeed we were unable to establish interaction-type estimates in the presence of an inhomogeneity.  In particular, our estimates prevent concentration only at a fixed point (in our case, the origin).  Nonetheless, they are sufficient in our setting, as we have obtained pre-compactness without the need for any moving spatial center.  It is worth noting that in Nakanishi's original work on power-type NLS, he is able to utilize non-interaction Morawetz estimates despite the fact that concentration may occur anywhere in space.  {In order to achieve this, he additionally used combinatorial arguments in the spirit of Bourgain (in the context of induction on energy) to essentially reduce to the case of concentration at a single point.}

The rest of the paper is organized as follows: In Section~\ref{S:prelim}, we introduce notation and collect various preliminary lemmas. In Section~\ref{S:LWP}, we discuss well-posedness and stability for \eqref{nls}.  In Section~\ref{S:Reduction}, we show that if Theorem~\ref{T} fails then we may construct a solution with pre-compact orbit in $H^1$.  In Section~\ref{S:Preclusion}, we then prove Morawetz estimates in the spirit of \cite{N} and use these estimates to preclude the possibility of such compact solutions, thus completing the proof of Theorem~\ref{T}.

\subsection*{Acknowledgements} L.B. and J.M. were supported by NSF grant DMS-2350225. J.M. was also supported by Simons Foundation grant MPS-TSM-00006622.

%%%%%%%%%%%%%%%%%%%%%%%
%%%%%%%%%%%%%%%%%%%%%%%
%%%%%%%%%%%%%%%%%%%%%%%
%%%%%%%%%%%%%%%%%%%%%%%
\section{Preliminaries}\label{S:prelim}
\subsection{Notation}
We write $A\lesssim B$ to denote that $A\leq CB$ for some $C>0$. We use the space-time Lebesgue norm
\[
\|u\|_{L_t^\alpha L_x^\beta(I\times \R^d)}=\|\|u(t,\cdot)\|_{L_x^\beta(\R^d)}\|_{L_t^\alpha(I)}.
\]
We assume familiarity with standard tools of harmonic analysis, such as Littlewood--Paley projections and associated estimates such as Bernstein estimates and the square function estimate. We use the standard notation $P_Nf$ or $f_N$ for the xprojection to frequency $N\in 2^{\mathbb{Z}}$. 

We denote the free Schr\"odinger propagator by $e^{it\Delta}$.  This is the Fourier multiplier operator with symbol $e^{-it|\xi|^2}$.  We recall the \emph{Duhamel formula} for solutions to $(i\partial_t + \Delta) u = F$:
\begin{equation}\label{Duham}
u(t)=e^{it\Delta}u_0-i\int_0^te^{i(t-s)\Delta}F(s)\,ds,\qtq{where} u_0=u|_{t=0}.
\end{equation}

\subsection{Basic Estimates}
We will need the Strichartz estimates for $e^{it\Delta}$.  We state the estimates in general space dimension $d$. 

A pair $(\alpha,\beta)$ is \textit{admissible} if $2\leq \alpha,\beta\leq \infty$, $\frac{2}{\alpha}+\frac{d}{\beta}=\frac{d}{2}$, and $(\alpha,\beta,d)\not=(2,\infty,2)$.

\begin{theorem}[Strichartz estimates \cite{GV, S2, KT}]
For any admissible pairs $(\alpha,\beta)$, $(\tilde \alpha,\tilde \beta)$ and $u:I\times\R^d\to\C$ a solution to $(i\partial_t+\Delta)u=F$ for some $F$, we have
\[
\|u\|_{L_t^\alpha L_x^\beta(I\times\R^d)}\lesssim \|u(t_0)\|_{L_x^2(\R^d)}+\|F\|_{L_t^{\tilde{\alpha}'}L_x^{\tilde{\beta}'}(I\times \R^d)}
\]
for any $t_0\in I$.  Here $'$ denotes the H\"older dual. 
\end{theorem}

We record here some basic pointwise estimates related to the nonlinearity.

\begin{lemma}\label{L:PW} Let $f(z)=|z|^p z$ and $u,w\in \C$.  Then we have the pointwise estimates
\begin{align}
|f(u+w)-f(u)| &\lesssim_p |w|^{p+1}+|w||u|^p, \label{E:BPW1} \\
% \begin{equation}\label{E:BPW2}
% |\nabla[f(u+w)-f(u)]| \lesssim_p |\nabla w| |u|^p + |\nabla w| |w|^p + |\nabla u| |w|^p. 
% \end{equation}
% Pretty sure above is typo, I redid the computation and checked photos I believe correct is
|\nabla[f(u+w)-f(u)]|&\lesssim_p|\nabla w||w|^p+|\nabla w||u|^p+|\nabla u||w|^p+|\nabla u||w||u|^{p-1}. \label{E:BPW2}
\end{align}
\end{lemma}

\begin{proof} The estimates follow from the identity
\[
f(u+w)-f(u) = w\int_0^1 f_z(u+\theta w)\,d\theta + \bar w\int_0^1 f_{\bar z}(u+\theta w)\,d\theta,
\]
together with the chain rule together and the estimates
\[
|f_z| + |f_{\bar z}| \lesssim_p |z|^p \qtq{and} |f_{\bar z z}|+|f_{zz}|+|f_{\bar z \bar z}| \lesssim |z|^{p-1}. 
\]
\end{proof}

Similar considerations yield the following additional nonlinear estimate.

\begin{lemma}\label{L:PW2} Let $f(z)=|z|^p z$ and $v_j:\R\to\R$.  Then
\[
\biggl|\sum_{j=1}^J \nabla[f(v_j)] - \nabla \biggl[f\biggl(\sum_{j=1}^J v_j\biggr)\biggr]\biggr| \lesssim_J \sum_{j\neq k} |\nabla v_j| |v_k|^p + \sum_{j\neq k} |\nabla v_j| |v_j|^{p-1} |v_k|. 
\]
\end{lemma}

We now define some exponents that will be used throughout the rest of the paper. Given $p\geq 2$, we define
\begin{equation}\label{def:rq}
r=r(p) = \begin{cases} p&p>4 \\ \frac{4p}{p-2} & 2\leq p\leq 4, \end{cases}%\quad q=q(p) = \begin{cases}\frac{4p}{p-2} & p>4 \\ \frac{8p}{p+2} & 2\leq p\leq 4.\end{cases}
\end{equation}
We use $L_t^{2p}L_x^{r(p)}$ as a scattering norm (cf. Section~\ref{S:LWP} below).  Observe by Sobolev embedding that
\begin{equation}\label{E:REstimate1}
\|u\|_{L_t^{2p} L_x^{r(p)}(I\times\R)}\lesssim \|u\|_{L_t^{2p}L_x^\frac{2p}{p-2}(I\times\R)}+\|\nabla u\|_{L_t^{2p}L_x^\frac{2p}{p-2}(I\times\R)},
\end{equation}
so that in particular by Strichartz
\begin{equation}\label{E:REstimate2}
\|e^{it\Delta}f\|_{L_t^{2p}L_x^{r(p)}(I\times\R)}\lesssim \|f\|_{H^1}.
\end{equation}

We define the regularity $s(p)$ by
\begin{equation}\label{sp}
s(p) = \begin{cases} \frac12-\frac2p & p>4 \\ \frac14-\frac1{2p} & 2\leq p\leq 4.\end{cases}
\end{equation}

We further define
\begin{equation}\label{def:rho}
\rho=\rho(p)=\begin{cases}
\infty & p>4\\
\frac{4}{6-p} & 2\leq p<4.
\end{cases}
\end{equation}
Throughout the paper, we will estimate $a$ and $\nabla a$ in $L_x^\rho(\R)$. 

We will frequently use the following basic nonlinear estimate, which follows from H\"older's inequality:

\begin{lemma}[Nonlinear Estimate]\label{L:StdNonlinEst} Let $f,g:I\times\R\to\C$ and $a:\R\to\R$. Then
\[
\|af|g|^p\|_{L_t^\frac{4}{3}L_x^1(I\times\R)}\lesssim \|a\|_{L_x^{\rho}}\|f\|_{L_t^4L_x^\infty(I\times\R)}\|g\|^p_{L_t^{2p}L_x^{r}(I\times\R)}.
\]
\end{lemma}

% A simple application of H\"older's inequality shows that if $f,g:I\times\R\to\C$ and $a:\R\to\R$ then
% \begin{equation}
% \|af|g|^p\|_{L_t^\frac{4}{3}L_x^1(I\times\R)}\lesssim \|a\|_{L_x^{\rho(p)}(\R)}\|f\|_{L_t^4L_x^\infty(I\times\R)}\|g\|_{L_t^{2p}L_x^{r(p)}(I\times\R)}^p
% \end{equation}
% This will serve as our primary nonlinear estimate.

Finally, we record a local smoothing estimate (cf. \cite{CS, Keraani, Sjo, Vega}) to be used in Section~\ref{S:Reduction}.
\begin{proposition}[Local smoothing]\label{P:LS} Let $K$ be a compact set in $\R\times\R$.  Then
\[
\|\nabla e^{it\Delta}\varphi\|_{L_{t,x}^2(K)} \lesssim_K \| |\nabla|^{\frac12}\varphi\|_{L^2}. 
\]
Consequently,
\[
\| \nabla e^{it\Delta}\varphi\|_{L_{t,x}^2(K)} \lesssim_K \| e^{it\Delta} \varphi\|_{L_t^{2p}L_x^r}^{\frac13} \|\nabla \varphi\|_{L^2}^{\frac23}. 
\]
\end{proposition}

\begin{proof} The first estimate is the standard local smoothing estimate for the linear Schr\"odinger equation \cite{CS, Sjo, Vega}).  To obtain the second statement, we let $N>0$ and treat low and high frequencies separately.  First, using H\"older's inequality followed by Bernstein estimates,
\begin{align*}
\| \nabla e^{it\Delta} P_{\leq N}\varphi\|_{L_{t,x}^2(K)} \lesssim_K N \|e^{it\Delta} \varphi\|_{L_t^{2p}L_x^r}.
\end{align*}
Next, by the local smoothing estimate and Bernstein estimates,
\begin{align*}
\|\nabla e^{it\Delta}P_{>N} \varphi\|_{L_{t,x}^2(K)} &\lesssim_K \| |\nabla|^{\frac12} \varphi\|_{L_x^2} \lesssim_K N^{-\frac12} \| \nabla \varphi\|_{L^2}.
\end{align*}
Optimizing in the choice of $N$ yields the result.
\end{proof}

%%%%%%%%%%%%%%%%%%%%%%%%%%%%%%%%%%%%%%
\subsection{Concentration compactness} 

We will need the following linear profile decomposition adapted to the Strichartz estimate, which plays a crucial role in Section~\ref{S:Reduction} (the reduction to compact solution).  This result follows more or less directly from the arguments in \cite{FXC}, although we prove vanishing of the remainders in a different norm than the one considered in \cite{FXC}.  

In what follows, the energy is defined as in \eqref{MEdef} and $a$ is assumed to be \emph{admissible} (in the sense that it satisfies the hypotheses of Theorem~\ref{T}).  We recall that $r(p)$ is the exponent defined in \eqref{def:rq}.

\begin{proposition}[Linear profile decomposition, \cite{FXC}]\label{P:LPD} Let $\{f_n\}$ be a bounded sequence in $H^1$.  Then there exists $J^*\in\{0,1,\dots,\infty\}$, nonzero profiles $\{\phi^j\}_{j=1}^{J^*}\subset H^1$, and space-time sequences $\{(t_n^j, x_n^j)\}_{j=1}^{J^*}\subset\R\times\R$ such that the following decomposition holds for each finite $J$:
\[
f_n(x) = \sum_{j=1}^{J} e^{it_n^j\Delta} \phi^j(x-x_n^j) + w_n^J(x),
\]
with $\{w_n^J\}$ bounded in $H^1$.  Moreover, the following properties hold: 
\begin{itemize}
\item The remainders $w_n^J$ vanish in the sense
\[
\lim_{J\to J^*}\lim_{n\to\infty} \| e^{it\Delta} w_n^J \|_{L_t^{2p} L_x^{r(p)}(\R\times\R)} = 0. 
\]
\item For each $j$, we have either $t_n^j\equiv 0$ or $|t_n^j|\to\infty$.  The same statement holds for the translation parameters $x_n^j$.  Furthermore, the profiles are asymptotically orthogonal in the sense that
\[
\lim_{n\to\infty} \bigl\{|t_n^j - t_n^k| + |x_n^j -x_n^k| \bigr\} = 0 \qtq{for any}j\neq k. 
\]
\item We have the following decoupling properties for any $J$:
\begin{align*}
M(f_n) = \sum_{j=1}^J M(\phi^j) + M(w_n^J) + o_n(1) \qtq{as} n\to\infty, \\
E(f_n) = \sum_{j-1}^J E(e^{it_n^j\Delta}\phi^j)+ E(w_n^J) + o_n(1) \qtq{as}n\to\infty. 
\end{align*}
\end{itemize}
\end{proposition}

As mentioned above, the proof follows from the arguments in \cite{FXC}.  Thus we will only provide a sketch of the proof here and refer the reader to \cite{FXC} for further details.  We further refer the reader to \cite{KV, V} for high-quality exposition of these techniques.  The key is to identify `bubbles of concentration' in the sequence, which are subsequently removed until the remainder vanishes in the sense described in the statement of the proposition.  To extract an individual bubble, one must identify a physical scale and a point in space-time at which concentration occurs.  Because we work with $H^1$-data and the norm for the remainder scales like some $\dot H^s$ for $s\in(0,1)$, we can ultimately set the scale parameters equal to one.  

The key to identifying a physical scale is the following lemma. 

\begin{lemma}[Refined Sobolev embedding]\label{L:RSE} Let $2<r<\infty$.  Then there exists some $\theta=\theta(r)$ such that for any $f\in H^1(\R)$,
\[
\|f\|_{L^r}\lesssim \sup_N \|P_N f\|_{L^r}^\theta \|f\|_{H^1}^{1-\theta}.
\]
\end{lemma}
\begin{proof}
For $2<r\leq 4$ we estimate using the Littlewood--Paley square function estimate, concavity of fractional powers, and Bernstein estimates:
\begin{align*}
\|f\|_{L^r}^r%&\lesssim \int\left(\sum_{N\in2^\Z}|f_n|^2\right)^\frac{r}{2}\,dx\\
%&\lesssim\int\left(\sum_{N\in2^\Z}|f_N|^2\right)^\frac{r}{4}\left(\sum_{N\in 2^\Z}|f_N|^2\right)^\frac{r}{4}\,dx\\
&\lesssim\sum_{N_1\leq N_2}\int|f_{N_1}|^\frac{r}{2}|f_{N_2}|^\frac{r}{2}\,dx\\
&\lesssim\sum_{N_1\leq N_2}\|f_{N_1}\|_{L^\frac{2r}{4-r}}\|f_{N_1}\|_{L^r}^\frac{r-2}{2}\|f_{N_2}\|_{L^r}^\frac{r-2}{2}\|f_{N_2}\|_{L^2}\\
&\lesssim\sup_N\|f_N\|_{L^r}^{r-2}\sum_{N_1\leq N_2}N_1^{-[\frac{r-2}{2r}]}N_2^{-[\frac{r-2}{2r}]}\||\nabla|^\frac{r-2}{2r}f_{N_1}\|_{L^\frac{2r}{4-r}}\|f_{N_2}\|_{L^2}\\
&\lesssim\sup_N\|f_N\|_{L^r}^{r-2}\sum_{N_1\leq N_2}\bigl(\tfrac{N_1}{N_2}\bigr)^\frac{r-2}{2r}\||\nabla|^\frac{r-2}{2r}f_{N_1}\|_{L^2}\||\nabla|^\frac{r-2}{2r}f_{N_2}\|_{L^2}\\
&\lesssim\sup_N\|f_N\|_{L^r}^{r-2}\||\nabla|^\frac{r-2}{2r}f\|_{L^2}^2.
\end{align*}

For $4<r<\infty$ we instead first use the Gagliardo--Nirenberg type estimate
\[
\|f\|_{L^r} \lesssim \|f\|_{L^4}^{\frac23+\frac{4}{3r}}\|\nabla f\|_{L^2}^{\frac13-\frac4{3r}}
%\|f\|_{L^r}\lesssim \|f\|_{L^4}^\frac{4}{r}\|f\|^\frac{r-4}{r}_{\dot H^1}
\]
and then estimate the $L^4$-norm as we did above. 
% and from here we estimate as above
% \begin{align*}
% \|f\|_{L^4}^4&\lesssim\int\left(\sum_{N\in 2^\Z}|f_N|^2\right)^2\,dx\\
% &\lesssim\sum_{N_1\leq N_2}\int|f_{N_1}|^2|f_{N_2}|^2\,dx\\
% &\lesssim\sum_{N_1\leq N_2}\|f_{N_1}\|_{L^\frac{2r}{r-4}}\|f_{N_1}\|_{L^r}\|f_{N_2}\|_{L^r}\|f_{N_2}\|_{L^2}\\
% &\lesssim\left(\sup_{N}\|f_N\|_{L^r}^2\right)\sum_{N_1\leq N_2}(N_1N_2)^{-\left(\frac{r-2}{2r}\right)}\||\nabla|^\frac{r-2}{2r}f_{N_1}\|_{L^\frac{2r}{r-4}}\||\nabla|^\frac{r-2}{2r}f_{N_2}\|_{L^2}\\
% &\lesssim\left(\sup_{N}\|f_N\|_{L^r}^2\right)\sum_{N_1\leq N_2}\left(\frac{N_1}{N_2}\right)^\frac{r-2}{2r}\||\nabla|^\frac{r-2}{2r}f_{N_1}\|_{L^2}\||\nabla|^\frac{r-2}{2r}f_{N_2}\|_{L^2}\\
% &\lesssim\left(\sup_N\|f_N\|_{L^r}^2\right)\||\nabla|^\frac{r-2}{2r}f\|_{L^2}^2,
% \end{align*}
% which gives the result.
\end{proof}

Using the refined Sobolev embedding, we can extract bubbles of concentration as follows.  From this point forward, we fix $r=r(p)$ as in \eqref{def:rq}. 

\begin{lemma}[Inverse Strichartz]\label{L:ISE}
Let $\{f_n\}$ be a sequence in $H^1$ satisfying
\[
\lim_{n\to\infty}\|f_n\|_{H^1}=A\qtq{and}\lim_{n\to\infty}\|e^{it\Delta}f_n\|_{L_t^{2p}L_x^r}=\varepsilon>0.
\]
Passing to a subsequence, there exist $\phi\in H^1$ with
\begin{equation}\label{E:ISE1}
\|\phi\|_{H^1}\gtrsim A(\tfrac{\varepsilon}{A})^\alpha
%^{\frac{r}{\theta(r+2)}+\frac{1}{\theta}}
\end{equation}
for some $\alpha=\alpha(p,r)\in (0,1)$, as well as $(t_n,x_n)\in\R\times\R$ such that
\begin{align}
&e^{it_n\Delta}f_n(x+x_n)\rightharpoonup \phi(x)\qtq{weakly in } H^1, \label{E:ISE2} \\
&\liminf_{n\to\infty}\left\{\|f_n\|_{H^\lambda}^2-\|f_n-\phi_n\|_{H^\lambda}^2-\|\phi\|_{H^\lambda}^2\right\}=0,\label{E:ISE3}
\end{align}
for $\lambda\in\{0,1\}$. Furthermore,
\begin{equation}\label{E:ISE4}
\liminf_{n\to\infty}\left\{\|a\,f_n\|_{L_x^{p+2}}^{p+2}-\|a\,[f_n-\phi_n]\|_{L_x^{p+2}}^{p+2}-\|a\,\phi_n\|_{L_x^{p+2}}^{p+2}\right\}=0,
\end{equation}
where $\phi_n = e^{it_n\Delta}\phi(x-x_n)$.
\end{lemma}

\begin{proof} Passing to a subsequence, we may assume
\[
\|f_n\|_{H^1}>\tfrac{A}{2}\qtq{and} \|e^{it\Delta}f_n\|_{L_t^{2p}L_x^r}>\tfrac{\varepsilon}{2}.
\]
Next, we observe that by H\"older's inequality and Strichartz estimates we have the inequality
\begin{align*}
\|e^{it\Delta}f\|_{L_t^{2p}L_x^{r}} &\lesssim \|e^{it\Delta}f\|_{L_t^\infty L_x^r}^{\theta_1}\|e^{it\Delta}f\|_{L_t^\frac{4r}{r-2}L_x^r}^{1-\theta_1}\lesssim \|e^{it\Delta}f\|_{L_t^\infty L_x^r}^{\theta_1}\|f\|_{H_x^1}^{1-\theta_1}
\end{align*}
with $\theta_1=\frac{p(r-2)-2r}{p(r-2)}$. Therefore, for all $n$ sufficiently large we may find times $t_n$ such that
\begin{equation}
A(\tfrac{\varepsilon}{A})^\frac{1}{\theta_1}\lesssim\|e^{it_n\Delta}f\|_{L_x^r}.
\end{equation}
Applying the Lemma~\ref{L:RSE}, we find some $\theta_2\in(0,1)$ and scales $N_n\in 2^\Z$ such that
\begin{equation}\label{E:ISEEqn1}
A(\tfrac{\varepsilon}{A})^\frac{1}{\theta_1\theta_2}\lesssim\|P_{N_n}(e^{it_n\Delta}f_n)\|_{L_x^r}.
\end{equation}
Finally, by interpolating between $L^2$ and $L^\infty$, we can obtain
\begin{equation}
A(\tfrac{\varepsilon}{A})^\frac{1}{\theta}\lesssim |[P_{N_n}e^{it_n\Delta}f_n](x_n)|.
\end{equation}
for some $x_n\in\R$ and some $\theta\in(0,1)$. 

We now note that because the sequence is bounded in $H^1$, the sequence $N_n$ is necessarily bounded above and below.  Indeed,  by \eqref{E:ISEEqn1} and Bernstein estimates we first estimate
\begin{align*}
A(\tfrac{\varepsilon}{A})^\frac{1}{\theta}\lesssim N_n^\frac{r-2}{2r}\|f_n\|_{L^2},\qtq{so that}(\tfrac{\varepsilon}{A})^\frac{2r}{\theta(r-2)}\lesssim N_n.
\end{align*}
Similarly,
\begin{align*}
A(\tfrac{\varepsilon}{A})^\frac{1}{\theta}\lesssim N_n^{-\frac{r+2}{2r}}\|\nabla f_n\|_{L^2},\qtq{so that}N_n\lesssim (\tfrac{A}{\varepsilon})^\frac{2r}{\theta(r+2)}
\end{align*}
As $N_n$ belongs to a discrete set, we can pass to a subsequence and assume $N_n\equiv M$ for some $M$.

We now define $g_n(x) = e^{it_n\Delta} f_n(x+x_n)$, which satisfies $\|g_n\|_{H^1}=\|f_n\|_{H^1}$.  Passing to a subsequence, we have that $g_n\rightharpoonup\phi$ weakly in $H^1$ for some $\phi$ (from which we can deduce \eqref{E:ISE2} and \eqref{E:ISE3}).  To obtain \eqref{E:ISE1} one can pair pair $\phi$ with the convolution kernel for $P_M$, which we may denote by $\psi$.  This yields
\[
M^{\frac12} \|\phi\|_{L^2} \gtrsim |\langle\phi,\psi\rangle| = \lim_{n\to\infty} |P_M[e^{it_n\Delta}(x_n)| \gtrsim A(\tfrac{\eps}{A})^{\frac{1}{\theta}},
\]
which (in light of the bounds on $M$) yields \eqref{E:ISE1}.

% To find the profile $\phi$, we define
% $$g_n(x)=e^{it_n\Delta}f_n(x+x_n),$$
% so that $\|g_n\|_{H^1}=\|f_n\|_{H^1}$ and Alaoglu's Theorem guarantees the existence of $\phi\in H^1$ with $g_n\rightharpoonup \phi$ weakly in $H^1$ giving \eqref{E:ISE2}. Similarly, we obtain \eqref{E:ISE3}  for automatically by the fact that the $H^\lambda$ are all Hilbert spaces paired with \eqref{E:ISE2}. To obtain \eqref{E:ISE1}, we pair $\phi$ with $\check{\psi}=P_{\tilde{N}}\delta_0$, the convolution kernel for $P_{\tilde{N}}$ which yields
% \begin{align*}
% |\langle \phi,\check{\psi}\rangle|&=\lim_{n\to\infty}\left|\int e^{it_n\Delta}f_n(x+x_n)\check{\psi}(x)\,dx\right|\\
% &=\lim_{n\to\infty}\left|\int e^{it_n\Delta}f_n(x)\check{\psi}(x-x_n)\,dx\right|\\
% &=\lim_{n\to\infty}|P_n[e^{it_n\Delta}f_n](x_n)|\gs A\left(\frac{\varepsilon}{A}\right)^\frac{1}{\theta}.
% \end{align*}
% On the other hand, we have by H\"older's inequality that
% $$|\langle \phi,\check{\psi}\rangle|\leq \|\check{\psi}\|_{L_x^2}\|\phi\|_{L_x^2}\lesssim N^{\frac{1}{2}}\|\phi\|_{L_x^2},$$
% so that we may conclude
% \[
% \|\phi\|_{L_x^2}\gtrsim N^{-\frac{1}{2}}A\left(\frac{\varepsilon}{A}\right)^\frac{1}{\theta}\gtrsim A\left(\frac{\varepsilon}{A}\right)^{\frac{r}{\theta(r+2)}+\frac{1}{\theta}},
% \]
% which gives \eqref{E:ISE1}. 

Finally, the potential energy decoupling follows from arguments as in \cite{FXC} (see also \cite{V}), relying on the refined Fatou Lemma of Br\'ezis and Lieb.  \end{proof}

Iterating the inverse Strichartz estimate leads to the profile decomposition.  We briefly sketch the ideas.

\begin{proof}[Sketch of the proof of Proposition~\ref{P:LPD}] We let $w_n^0 = f_n$.  If $\|e^{it\Delta} w_n^0\|_{L_t^{2p} L_x^r} \to 0$, we set $J^*=0$ and we are done.  Otherwise, we apply Lemma~\ref{L:ISE} to obtain the parameters $(t_n^1,x_n^1)$ and the profile $\varphi^1$.  We then let $w_n^1 = w_n^0 - e^{-it_n^1\Delta}\varphi^1(x-x_n^1)$.  If $\|e^{it\Delta} w_n^1\|_{L_t^{2p} L_x^r}\to 0$, then we set $J^*=1$ and stop.  Otherwise, we continue to apply Lemma~\ref{L:ISE}. Proceeding in this fashion leads to the construction of the profiles and remainders, which can then be shown to satisfy the mass/energy decoupling and the vanishing remainder.  Finally, one uses the weak convergence property to establish the asymptotic orthogonality of profiles.  In the case that $\{t_n^j\}_{n}$ is bounded or $\{x_n^j\}_{n}$ is bounded for some $j$. we can reduce to $t_n^j\equiv 0$ or $x_n^j\equiv 0$ by incorporating a translation into the definition of the profile.  For complete details, we again refer the reader to \cite{FXC,V}. 
\end{proof}

\section{Local well-posedness and stability}\label{S:LWP}

As described in the introduction, standard arguments (using Strichartz estimates and a contraction mapping argument to prove the existence of a solution to the Duhamel formula) yield local well-posedness for \eqref{nls} for $a$ satisfying $a,\nabla a\in L^\infty$.  The defocusing assumption and conservation of mass and energy then allow us to obtain global existence.  The scattering statement relies on the nonlinear estimate Lemma~\ref{L:StdNonlinEst}, and in particular this requires decay assumptions on $a$ in the range $2<p\leq 4$. 

\begin{proposition}[Well-posedness] Fix an admissible, nonnegative $a$ and $p\geq 2$. 

\begin{itemize}
\item[(i)] \textbf{Initial-value problem.} For any $u_0\in H^1(\R)$, there exists a unique global solution to \eqref{nls} with $u|_{t=0}=u_0$.
\item[(ii)] \textbf{Final-state problem.} For any $u_+\in H^1(\R)$, there exists a unique global solution $u$ to \eqref{nls} such that
\[
\lim_{t\to\infty} \|u(t)-e^{it\Delta}u_+\|_{H^1} = 0.
\] 
A similar statement holds backward in time. 
\end{itemize}

For both the initial-value problem and final-state problem, the solution is uniformly bounded in $H^1$.  Moreover, if
\[
\|u\|_{L_t^{2p} L_x^r(\R\times\R)}<\infty
\]
for some $r\in[p,\infty]$, then $u$ scatters in $H^1$ as $t\to\pm\infty$.
\end{proposition}

We record the following corollary (the standard small-data scattering result).  Note that because of the defocusing assumption $a\geq 0$, the $H^1$-norm of a solution is controlled by its mass and energy.

\begin{corollary} Let $u$ be a solution to \eqref{nls}. If $M(u)+E(u)$ is sufficiently small, then $u$ scatters in $H^1$ as $t\to\pm\infty$, with
\[
\|u\|_{L_t^{2p} L_x^{r(p)}(\R\times\R)} \lesssim \|u_0\|_{H^1}.
\]
\end{corollary}

In the rest of the paper we will take $r=r(p)$ and $\rho=\rho(p)$ as defined in \eqref{def:rq} and \eqref{def:rho}.

We will also need a stability theory for \eqref{nls}.  The arguments are fairly standard and are similar in spirit to those used to obtain local well-posedness.  Thus we will only sketch the proofs, and we refer the reader to \cite{KV,V} for detailed exposition.  In what follows, we assume $a$ satisfies the assumptions of Theorem~\ref{T}. 

\begin{proposition}[Stability]\label{P:stability} Let $t_0\in I$, $u_0\in H^1$, and suppose $\tilde{u}:I\times\R\to \C$ is a solution to the equation
\[
(i\partial_t+\Delta)\tilde{u}=a(x)|\tilde{u}|^p\tilde{u}+e
\]
for some function $e$. Further assume
\[
\|\tilde{u}\|_{L_t^{2p}L_x^{r}(I\times\R)}=L<\infty.
\]
Then there exist $\varepsilon_1=\varepsilon_1(L)$ such that for all $0<\varepsilon\leq\varepsilon_1$, if
\[
\|u_0-\tilde{u}(t_0)\|_{H_x^1}+\|e\|_{L_t^\frac{4}{3}L_x^1(I\times\R)}+\|\nabla e\|_{L_t^\frac{4}{3}L_x^1(I\times\R)}<\varepsilon,
\]
then there exists a true solution $u:I\times\R\to\C$ to \eqref{nls} with $u(t_0)=u_0$ satisfying
\[
\|u-\tilde{u}\|_{L_t^{2p}L_x^{r}(I\times\R)}+\|(\nabla)^s(u-\tilde{u})\|_{L_t^\alpha L_x^\beta(I\times\R)}\lesssim_{L}\varepsilon,
\]
for any admissible pair $(\alpha,\beta)$ and $s\in\{0,1\}$.
\end{proposition}

Proposition~\ref{P:stability} is obtained by iterating the following short-time result.

\begin{proposition}[Short-time perturbations]
Let $t_0\in I$, $u_0\in H^1$, and suppose $\tilde{u}:I\times\R\to\C$ is a solution to the equation
$$(i\partial_t+\Delta)\tilde{u}=a(x)|\tilde{u}|^p\tilde{u}+e,$$
for some function $e$. Then there exist $\varepsilon_0$, $\delta$ such that for all $0<\varepsilon\leq\varepsilon_0$ if
\[
\|\tilde{u}\|_{L_t^{2p}L_x^r(I\times\R)}+\|\tilde{u}\|_{L_t^4L_x^\infty(I\times\R)}+\|\nabla \tilde{u}\|_{L_t^4L_x^\infty(I\times\R)}<\delta
\]
and
\[
\|u_0-\tilde{u}(t_0)\|_{H_x^1}+\|e\|_{L_t^\frac{4}{3}L_x^1(I\times\R)}+\|\nabla e\|_{L_t^\frac{4}{3}L_x^1(I\times\R)}<\varepsilon,
\]
then there exists a solution $u:I\times\R\to \C$ to \eqref{nls} with $u(t_0)=u_0$ satisfying
\begin{equation}\label{E:StabError}
\|(\nabla)^s[a(x)\{|u|^pu-|\tilde{u}|^p\tilde{u}\}]\|_{L_t^\frac{4}{3}L_x^1(I\times\R)}\lesssim\varepsilon 
\end{equation}
for $s\in\{0,1\}$. Furthermore,
\[
\|u-\tilde{u}\|_{L_t^{2p}L_x^r(I\times\R)}+\|(\nabla)^s[u-\tilde{u}]\|_{L_t^\alpha L_x^\beta(I\times\R)}\lesssim \varepsilon,
\]
for all admissible $(\alpha,\beta)$ and $s\in\{0,1\}$.
\end{proposition}

\begin{proof}
By time translation we may assume without loss of generality that $t_0=0=\inf I$. Let $u:I\times\R\to \C$ be the solution to \eqref{nls} with $u|_{t=0}=u_0$. Now denote $f(z)=|z|^pz$ and define $w=u-\tilde{u}$ so that 
\[
(i\partial_t+\Delta)w=a(x)f(w+\tilde u)-a(x)f(\tilde{u})-e.
\]
We let
\[
A_s(t)=\|(\nabla)^s[a\,f(u)-a\,f(\tilde{u})]\|_{L_t^\frac{4}{3}L_x^1((0,t)\times\R)}\qtq{and} A(t)=A_0(t)+A_1(t).
\]

For any $4\leq q\leq\infty$, we may use Strichartz estimates to obtain
\begin{align*}
\|(\nabla)^sw\|_{L_t^qL_x^\frac{2q}{q-4}((0,t)\times\R)}&\lesssim \|u_0-\tilde{u}(0)\|_{H_x^s}+A_s(t)+\|(\nabla)^se\|_{L_t^\frac{4}{3}L_x^1((0,t)\times\R)},\\
\end{align*}
so that on $(0,t)$ we have
\begin{equation}\label{StabEqn1}
\|(\nabla)^sw\|_{L_t^qL_x^\frac{2q}{q-4}}\lesssim \varepsilon+A_s(t).
\end{equation}
Additionally, by \eqref{E:REstimate1} and \eqref{StabEqn1} we have that
\begin{equation}\label{StabEqn2}
\|w\|_{L_t^{2p}L_x^{r}}\lesssim \|w\|_{L_t^{2p}L_x^\frac{2p}{p-2}}+\|\nabla w\|_{L_t^{2p}L_x^\frac{2p}{p-2}}\lesssim \varepsilon +A(t).
\end{equation}

% We can now estimate via the pointwise esimate 
% \[
% |f(z+w)-f(z)|\lesssim |w|^{p+1}+|w||z|^p,
% \] 

Now, by the nonlinear estimate Lemma~\ref{L:StdNonlinEst}, \eqref{E:BPW1}, \eqref{StabEqn1}, and \eqref{StabEqn2},
\begin{align*}
A_0(t)&=\|a[f(\tilde{u}+w)-f(\tilde{u})]\|_{L_t^\frac{4}{3}L_x^1((0,t)\times\R)}\\
&\lesssim\|a[|w|^{p+1}+w|\tilde{u}|^p]\|_{L_t^\frac{4}{3}L_x^1((0,t)\times\R)}\\
&\lesssim \|a\|_{\rho}\|w\|_{L_t^4L_x^\infty((0,t)\times\R)}\bigl(\|w\|_{L_t^{2p}L_x^r((0,t)\times\R)}^p+\|\tilde{u}\|_{L_t^{2p}L_x^r((0,t)\times\R)}^p\bigr)\\
&\lesssim_a(\varepsilon+A(t))^{p+1}+\delta^p(\varepsilon+A(t))
\end{align*}
Similarly, using the Lemma \ref{L:PW}, the nonlinear estimate Lemma~\ref{L:StdNonlinEst}, \eqref{StabEqn1}, and \eqref{StabEqn2} over $(0,t)\times\R$, we can obtain
\[
A_1(t)\lesssim (\varepsilon+A(t))^{p+1}+\delta^p(\varepsilon+A(t))+\delta(\varepsilon+A(t))^p.
\]
% \begin{align*}
% A_1(t)&=\|\nabla[a\{f(\tilde{u}+w)-f(\tilde{u})\}]\|_{L_t^\frac{4}{3}L_x^1}\\
% &\lesssim \|\nabla a[f(\tilde{u}+w)-f(\tilde{u})]\|_{L_t^\frac{4}{3}L_x^1}\\
% &\quad+\|a[\nabla\{f(\tilde{u}+w)-f(\tilde{u})\}]\|_{L_t^\frac{4}{3}L_x^1}\\
% &\lesssim_a A_0(t)+\|\nabla w\|_{L_t^4L_x^\infty}\Bigl(\|w\|_{L_t^{2p}L_x^r}^p+\|u\|_{L_t^{2p}L_x^r}^p\Bigr)\\
% &\quad+\|\nabla u\|_{L_t^4L_x^\infty}\|w\|_{L_t^{2p}L_x^r}\Bigl(\|w\|_{L_t^{2p}L_x^r}^{p-1}+\|u\|_{L_t^{2p}L_x^r}^{p-1}\Bigr)\\
% &\lesssim(\varepsilon+A(t))^{p+1}+\delta^p(\varepsilon+A(t))+\delta(\varepsilon+A(t))^p.
% \end{align*}
Combining these we obtain
\begin{align*}
A(t)&\lesssim_a (\varepsilon+A(t))^{p+1}+\delta^p(\varepsilon+A(t))+\delta(\varepsilon+A(t))^p,
\end{align*}
so that for $\delta$ and $\varepsilon$ sufficiently small we obtain $A(t)\lesssim \varepsilon$ for all $t\in I$, which completes the proof.
\end{proof}

As mentioned above, iterating the preceding result leads to Proposition~\ref{P:stability} (see e.g. \cite{KV,V}).

\section{The reduction to compact solutions}\label{S:Reduction} 

In this section we carry out one of the major steps to proving Theorem~\ref{T}.  In particular, we prove that if Theorem~\ref{T} fails, then we can find a solution to \eqref{nls} that has pre-compact orbit in $H^1$.  Throughout this section, we fix an inhomogeneity $a(x)$ satisfying the hypotheses of Theorem~\ref{T}.

We first define the increasing function 
\[
L:[0,\infty)\to[0,\infty]\qtq{by} L(E) = \sup\bigl\{ \|u\|_{L_t^{2p} L_x^{r}(\R\times\R)}\bigr\},
\]
where the supremum is taken over all solutions $u$ to \eqref{nls} with
\[
E(u)+M(u)<E. 
\]
By the small-data theory, we have
\[
L(E) \lesssim_E 1 \qtq{for} E<\eta_0,
\]
where $\eta_0=\eta_0(a)$ is the small-data threshold.  We let
\[
E_c = \sup\bigl\{E:L(E)<\infty\bigr\}. 
\]
To prove Theorem~\ref{T}, it suffices to prove that $E_c=\infty$.  We argue by contradiction and assume that $E_c<\infty$.  We will then prove the following: 
 
\begin{proposition}[Palais--Smale condition]\label{P:PS} Let $u_n:\R\times\R\to\C$ be a sequence of $H^1$ solutions to \eqref{nls} such that
\[
M(u_n)+E(u_n) \to E_c.
\]
Suppose that $\{t_n\}\subset\R$ are such that
\[
\lim_{n\to\infty} \|u_n\|_{L_t^{2p} L_x^r([t_n,\infty)\times\R)} = \lim_{n\to\infty} \|u_n\|_{L_t^{2p} L_x^r((-\infty,t_n]\times\R)} = \infty. 
\]
Then $\{u_n(t_n)\}$ converges along a subsequence in $H^1$. 
\end{proposition}

With Proposition~\ref{P:PS} in hand, we can readily deduce the existence of a compact solution.

\begin{corollary}\label{C:AP} Suppose Theorem~\ref{T} fails. Then there exists an $H^1$-solution $v:\R\times\R\to\C$ such that $\{v(t):t\in\R\}$ is pre-compact in $H^1$. 
\end{corollary}

\begin{proof}[Proof of Corollary~\ref{C:AP}] Suppose Theorem~\ref{T} fails, so that $E_c<\infty$. By definition, we may find a sequence of (scattering) solutions $u_n$ such that
\[
M(u_n)+E(u_n)\nearrow E_c \qtq{and} \lim_{n\to\infty} S_\R(u_n) = \infty. 
\]
We now choose $t_n\in\R$ such that
\[
S_{(-\infty,t_n]}(u_n) = S_{[t_n,\infty)}(u_n).
\]
Applying a time translation to each solution if necessary, we may assume $t_n\equiv 0$.  We may then apply Proposition~\ref{P:PS} to conclude that the sequence $\{u_n(0)\}$ converges strongly to some limit $\varphi \in H^1$ (up to a subsequence). We now let $v:\R\times\R\to\C$ be the solution to \eqref{nls} with $v|_{t=0}=\varphi$.  Observe that by stability (Proposition~\ref{P:stability}) we have
\[
M(v)+E(v) = E_c \qtq{and} S_{(-\infty,0]}(v) = S_{[0,\infty)}(v) = \infty. 
\]

We now claim that $\{v(t):t\in\R\}$ is pre-compact in $H^1$.  To this end,  we take an arbitrary sequence $\{s_n\}\subset\R$.  Noting that the scattering size of $v$ is finite on any finite time interval, we have that
\[
\lim_{n\to\infty} S_{[s_n,\infty)}(v) = \lim_{n\to\infty} S_{(-\infty,s_n]}(v) = \infty. 
\]
Thus we may apply Proposition~\ref{P:PS} to conclude that $\{v(s_n)\}$ converges along a subsequence in $H^1$. 
\end{proof}

It therefore remains to prove Proposition~\ref{P:PS}.  An essential ingredient in the proof is the following proposition, which allows us to construct scattering solutions to \eqref{nls} corresponding to initial data that is translated far from the origin.

\begin{proposition}[Scattering for far-away profiles]\label{P:SFAP} Let $\varphi\in H^1$.  Let $\{t_n\}$ satisfy $t_n\equiv 0$ or $t_n\to\pm\infty$, and let $\{x_n\}\subset\R$ satisfy $|x_n|\to\infty$.  For all sufficiently large $n$, there exists a global solution $v_n$ to \eqref{nls} with
\[
v_n(0,x) = \varphi_n(x) := e^{it_n\Delta}\varphi(x-x_n)
\]
that scatters in both time directions and obeys
\[
\|v_n\|_{L_t^{2p} L_x^{r}(\R\times\R)} +\|(\nabla)^sv_n\|_{L_t^\alpha L_x^\beta}\lesssim_{\|\varphi\|_{H^1}} 1,
\]
for all admissible $(\alpha,\beta)$ and $s\in\{0,1\}$. Furthermore, for any $\eps>0$ there exists $N$ and $\psi\in C_c^\infty(\R\times\R)$ such that for $n\geq N$,
\begin{align*}
&\|v_n - \psi(\cdot+t_n,\cdot-x_n)\|_{L_t^{2p} L_x^{r}(\R\times\R)} < \eps, \\
&\|(\nabla)^s[v_n-\psi(\cdot+t_n,\cdot-x_n)]\|_{L_t^\alpha L_x^{\beta}(\R\times\R)}<\eps,
\end{align*}
for all admissible $(\alpha,\beta)\not\in\{(\infty,2),(4,\infty)\}$ and $s\in\{0,1\}$.
\end{proposition}

\begin{proof}
Our strategy is to define approximate solutions to \eqref{nls} that obey space-time bounds and then apply stability (Proposition~\ref{P:stability}) to construct our scattering solution. The definition of the approximate solutions depends on the behavior of the inhomogeneity $a(x)$ as $x\to\pm\infty$. The repulsivity and nonnegativity assumptions on $a(x)$ guarantee that $a(x)\to a_\pm$ as $x\to\pm\infty$ for $a_\pm\geq 0$. When $2<p\leq 4$ we necessarily have $a_\pm=0$, however for $p>4$ this need not be the case.

We will first show in detail how to treat the case $a_\pm=0$. We define a sequence of smooth cutoff functions satisfying
$$\chi_n(x)=\begin{cases}
1 & |x+x_n|>\frac{1}{2}|x_n|\\
0 & |x+x_n|<\frac{1}{4}|x_n|
\end{cases}$$
that obey the symbol bounds $|\chi_n^{(k)}|\lesssim |x_n|^{-k}$ for all $k\geq 0$. We note that the $\chi_n\to 1$ pointwise as $n\to\infty$.  We next define the frequency projections $P_n=P_{\leq |x_n|^\theta}$ for some $0<\theta\ll1$. 

Given $T\geq 1$, we now define approximate solutions as follows: First,
\[
\tilde{v}_{n,T}(t,x)=\chi_n(x-x_n)e^{it\Delta}P_n\varphi(x-x_n)\qtq{for}|t|\leq T,
\]
while for $|t|>T$ we take
\[
\tilde{v}_{n,T}(t,x)=\begin{cases}
e^{i(t-T)\Delta}\tilde{v}_{n,T}(T) & t>T\\
e^{i(t+T)\Delta}\tilde{v}_{n,T}(-T) & t<T.
\end{cases}
\]

We will show that the $\tilde v_{n,T}$ satisfy
\begin{equation}\label{E:vn-intial-data}
\limsup_{T\to\infty}\limsup_{n\to\infty}\|\tilde{v}_{n,T}(t_n)-\varphi_n\|_{H_x^1}=0,
\end{equation}
\begin{equation}\label{E:vn-good-bounds}
\limsup_{T\to\infty}\limsup_{n\to\infty}\bigl\{|\tilde{v}_{n,T}\|_{L_t^{2p}L_x^r}+\|(\nabla)^s\tilde{v}_{n,T}\|_{L_t^\alpha L_x^\beta}\bigr\}\lesssim 1,
\end{equation}
for admissible $(\alpha,\beta)$ and $s\in \{0,1\}$, and finally that
\begin{equation}\label{E:vn-error-bounds}
\lim_{T\to\infty}\limsup_{n\to\infty} \|(\nabla)^s[(i\partial_t+\Delta)\tilde{v}_{n,T}-a(x)|\tilde{v}_{n,T}|^p\tilde{v}_{n,T}]\|_{L_t^\frac{4}{3}L_x^1}=0,
\end{equation}
for $s\in \{0,1\}$.
\begin{proof}[Proof of \eqref{E:vn-intial-data}]
First, suppose that $t_n\equiv0$. Then
\begin{equation*}
\|\tilde{v}_{n,T}(0)-\varphi_n\|_{H_x^1}=\|\chi_nP_n\varphi-\varphi\|\to0\qtq{as}n\to\infty
\end{equation*}
by the dominated convergence theorem. Next suppose $t_n\to\infty$ and fix $T>0$. Then for $n$ sufficiently large $|t_n|>T$ so that by the triangle inequality
\begin{align*}
\|\tilde{v}_{n,T}(t_n)-\varphi_n\|_{H_x^1}\leq \|(\chi_n-1)e^{iT\Delta}P_n\varphi\|_{H_x^1}+\|P_n\varphi-\varphi\|_{H_x^1},
\end{align*}
which again tends to $0$ as $n\to\infty$ by the dominated convergence theorem. The case that $t_n\to-\infty$ is similar.
\end{proof}

\begin{proof}[Proof of \eqref{E:vn-good-bounds}]
We will only show how to obtain the $L_t^{2p}L_x^r$-bounds, as the others are similar. On the time interval $[-T,T]$ we have by H\"older's inequality and \eqref{E:REstimate2} that
\[
\|\tilde{v}_{n,T}\|_{L_t^{2p}L_x^{r}([-T,T]\times\R)}\lesssim \|e^{it\Delta}P_n\varphi(x-x_n)\|_{L_t^{2p}L_x^{r}([-T,T]\times\R)}\lesssim \|\varphi\|_{H_x^1}.
\]
Similarly, by \eqref{E:REstimate2} for $t>T$ we have that
\begin{align*}
\|\tilde{v}_{n,T}\|_{L_t^{2p}L_x^{r}((T,\infty)\times\R)}&=\|e^{i(t-T)\Delta}\tilde{v}_{n,T}(T)\|_{L_{t}^{2p}L_x^{r}((T,\infty)\times\R)}\lesssim\|\tilde{v}_{n,T}(T)\|_{H_x^1},
\end{align*}
with a similar estimate for $t<-T$.  Thus it suffices to estimate $\tilde v_{n,T}$ in $L_t^\infty H_x^1$.  To this end, we estimate using H\"older's inequality and Strichartz estimates
\begin{align*}
\|\nabla\tilde{v}_{n,T}\|_{L_t^\infty L_x^2}&\lesssim \|\nabla \chi_n\|_{L_x^\infty}\|e^{it\Delta}P_n\varphi(x-x_n)\|_{L_t^\infty L_x^2([-T,T]\times\R)}\\
&\quad+\|\chi_n\|_{L_x^\infty}\|e^{it\Delta}P_n\nabla\varphi(x-x_n)\|_{L_t^\infty L_x^2([-T,T]\times\R)}\lesssim\|\varphi\|_{H_x^1}.
\end{align*}
As the estimate without the derivative is simpler, this completes the proof. \end{proof}

\begin{proof}[Proof of \eqref{E:vn-error-bounds}] We will focus on the more difficult case $s=1$, which will also demonstrate how to treat the case $s=0$. We set
\[
e_{n,T}=(i\partial_t+\Delta)\tilde{v}_{n,T}-a(x)|\tilde{v}_{n,T}|^p\tilde{v}_{n,T}.
\]
Observe that for $|t|>T$ we have
\[
e_{n,T}=-a(x)|\tilde{v}_{n,T}|^p\tilde{v}_{n,T}.
\]
Thus on $(T,\infty)\times\R$ we have by the nonlinear estimate Lemma~\ref{L:StdNonlinEst}
\begin{align*}
\|\nabla e_{n,T}\|_{L_t^\frac{4}{3}L_x^1}&\lesssim\|\nabla a\|_{L^\rho}\|\tilde{v}_{n,T}\|_{L_t^4L_x^\infty}\|\tilde{v}_{n,T}\|^p_{L_t^{2p}L_x^{r}}+\|a\|_{L^\rho}\|\nabla\tilde{v}_{n,T}\|_{L_t^{4}L_x^\infty}\|\tilde{v}_{n,T}\|^p_{L_t^{2p}L_x^{r}}.
\end{align*}
By \eqref{E:vn-good-bounds} and the dominated convergence theorem, we note that the $L_t^{2p} L_x^r$ norm of $\tilde v_{n,T}$  on $(T,\infty)\times\R$ tends to zero as $n,T\to\infty$, so that the quantity above does as well.  The case of $t<-T$ is similar.

It remains to estimate the error on the region $|t|<T$. In this region, we have
\begin{align}
e_{n,T}(t,x)&=\Delta\chi_n(x-x_n)\,e^{it\Delta}P_n\varphi(x-x_n)\label{E:Error1}\\
&\ +2\nabla\chi_n(x-x_n)\cdot\nabla e^{it\Delta}P_n\varphi(x-x_n)\label{E:Error2}\\
&\ -a(x)\chi_n(x-x_n)^{p+1}|e^{it\Delta}P_n\varphi(x-x_n)|^pe^{it\Delta}P_n\varphi(x-x_n).\label{E:Error3}
\end{align}
We now estimate the contributions of each of these terms. In what follows all space-time norms take place over $[-T,T]\times\R$ and space norms over $\R$. Starting with \eqref{E:Error1} we have by H\"older's inequliaty, the symbol bounds on the $\chi_n$, and Strichartz estimates
\begin{align*}
\|\nabla[\Delta&\chi_n(x-x_n)e^{it\Delta}P_n\varphi(x-x_n)]\|_{L_t^\frac{4}{3}L_x^1}\\
&\lesssim|T|^\frac{3}{4}\|(\nabla\Delta\chi_n(x-x_n)\|_{L_x^2}\|e^{it\Delta}P_n\varphi(x-x_n)\|_{L_t^\infty L_x^2}\\
&\quad+|T|^\frac{3}{4}\|\Delta\chi_n(x-x_n)\|_{L_x^2}\|e^{it\Delta}P_n\nabla\varphi(x-x_n)\|_{L_t^\infty L_x^2}\\
&\lesssim (|x_n|^{-\frac{5}{2}}+|x_n|^{-\frac{3}{2}})|T|^\frac{3}{4}\|\varphi\|_{H_x^1}\to 0\qtq{as}n\to\infty.
\end{align*}
Next we turn to \eqref{E:Error2}, where we similarly estimate by H\"older's ineuqliaty, symbol bounds, Strichartz, and Bernstein estimates
\begin{align*}
\|\nabla[\nabla(&\chi_n(x-x_n)\cdot\nabla e^{it\Delta}P_n\varphi(x-x_n)]\|_{L_t^\frac{4}{3}L_x^1}\\
&\lesssim\|\Delta(\chi_n(x-x_n))\nabla e^{it\Delta}P_n\varphi(x-x_n)\|_{L_t^\frac{4}{3}L_x^1}\\
&\qquad+\|\nabla \chi_n(x-x_n)\Delta e^{it\Delta}P_n\varphi(x-x_n)\|_{L_t^\frac{4}{3}L_x^1}\\
&\lesssim |T|^\frac{3}{4}(|x_n|^{-\frac{3}{2}}\|\varphi\|_{H_x^1}+|x_n|^{-\frac{1}{2}}\|P_n\Delta\varphi(x-x_n)\|_{L_x^2})\\
&\lesssim|T|^\frac{3}{4}(|x_n|^{-\frac{3}{2}}+|x_n|^{-\frac{1}{2}+\theta})\|\varphi\|_{H_x^1}\to0\qtq{as}n\to\infty.
\end{align*}
Finally, this leaves \eqref{E:Error3}, which we estimate via the standard nonlinear estimate Lemma~\ref{L:StdNonlinEst} and Strichartz estimates: 
\begin{align*}
\|\nabla&[a(x)\chi_n(x-x_n)^{p+1}|e^{it\Delta}P_n\varphi(x-x_n)|^pe^{it\Delta}P_n\varphi(x-x_n)]\|_{L_t^\frac{4}{3}L_x^1}\\
&\lesssim \|\nabla a\|_{L_x^\rho(|x|>\frac{1}{4}|x_n|)}\|e^{it\Delta}P_n\varphi\|_{L_t^4L_x^\infty}\|e^{it\Delta}P_n\varphi\|_{L_t^{2p}L_x^r}^p\\
&\quad+\|a\|_{L_x^\rho(|x|>\frac{1}{4}|x_n|)}\|e^{it\Delta}\nabla P_n\varphi\|_{L_t^4L_x^\infty}\|e^{it\Delta}P_n\varphi\|_{L_t^{2p}L_x^r}^p\\
&\lesssim \Bigl(\|\nabla a\|_{L_x^\rho(|x|>\frac14|x_n|)}+\|a\|_{L_x^\rho(|x|>\frac14|x_n|)}\Bigr)\|\varphi\|_{H_x^1}^{p+1}.
\end{align*}
We now observe that this quantity tends to zero using the dominated convergence theorem and the fact that $a_\pm=0$.\end{proof}
%If $\rho<\infty$ the above tends to zero by the dominated convergence theorem. Similarly if $\rho=\infty$ because $a_{\pm}=\lim_{x\to\pm\infty}a(x)=0$ $a(x)\to 0$ the above tends to zero.

With \eqref{E:vn-intial-data}--\eqref{E:vn-error-bounds} in hand, we see that $\tilde{v}_{n,T}$ satisfy the assumptions of Proposition~\eqref{P:stability} for sufficiently large $n$ and $T$.  Thus, after applying a time translation, for $n$ sufficiently large there exist solutions $v_n$ to \eqref{nls} with $v_n(0)=\varphi_n(x)$ satisfying the desired space-time bounds. 

It remains to show the approximation by $C_c^\infty$ functions. Again, we only discuss the case of approximation in $L_t^{2p}L_x^r$ as the others are similar. By construction we have that
\[
\lim_{T\to\infty}\limsup_{n\to\infty}\|v_n(\cdot-t_n)-\tilde{v}_{n,T}(\cdot)\|_{L_t^{2p}L_x^{r}(\R\times\R)}=0.
\]
Thus, choosing $\psi\in C_c^\infty(\R\times\R)$ satisfying
\[
\|e^{it\Delta}\varphi-\psi\|_{L_t^{2p}L_x^{r}(\R\times\R)}<\varepsilon,
\]
it now suffices to show that
\begin{equation}\label{E:vn-approx}
\|\tilde{v}_{n,T}-e^{it\Delta}\varphi(x-x_n)\|_{L_t^{2p}L_x^{r}(\R\times\R)}<\varepsilon,
\end{equation}
for sufficiently large $n$ and $T$.  For $|t|<T$ we have
\begin{align*}
\|\tilde{v}_{n,T}&-e^{it\Delta}\varphi(x-x_n)\|_{L_t^{2p}L_x^{r}((-T,T)\times\R)}\\
&\lesssim\|\chi_ne^{it\Delta}P_n\varphi-e^{it\Delta}\varphi\|_{L_t^{2p}L_x^{r}((-T,T)\times\R)}\\
&\lesssim\|(\chi_n-1)e^{it\Delta}P_n\varphi\|_{L_t^{2p}L_x^{r}((-T,T)\times\R)}+\|P_n\varphi-\varphi\|_{H_x^1}=o(1)\qtq{as}n\to\infty.
\end{align*}
For $t>T$ we have by \eqref{E:REstimate2}
\begin{align*}
\|\tilde{v}_{n,T}&-e^{it\Delta}\varphi(x-x_n)\|_{L_t^{2p}L_x^{r}((T,\infty)\times\R)}\\
&\lesssim \|\chi_ne^{iT\Delta}P_n\varphi-e^{iT\Delta}\varphi\|_{H^1}\\
&\lesssim \|(\chi_n-1)e^{iT\Delta}P_n\varphi\|_{H^1}+\|P_n\varphi-\varphi\|_{H^1}=o(1)\qtq{as}n\to\infty.
\end{align*}
The case $t<-T$ is similar. Thus we obtain \eqref{E:vn-approx}, and complete the proof of Proposition~\ref{P:SFAP} in the case that $a_\pm=0$.

We now sketch the details of the case where $a_\pm>0$. Without loss of generality, assume $x_n\to\infty$.  For $a_+=0$, we used the linear Schr\"odinger equation as an approximate model for \eqref{nls}.  In the case $a_+>0$, we instead use the standard power-type NLS 
\begin{equation}\label{nlspow}
(i\partial_t+\Delta)u=a_+|u|^pu.
\end{equation}

We first define smooth cutoff functions
\[
\phi_n(x)=\begin{cases}
1 & x+x_n>\frac{1}{2}x_n\\
0 & x+x_n<\frac{1}{4}x_n
\end{cases}
\]
obeying the same symbol bounds $|\phi_n^{(k)}|\lesssim |x_n|^{-k}$ for all $k\geq 0$. These functions are similar to the $\chi_n$, but only cutoff to one side, reflecting the fact that the $x_n\to\infty$. 

If $t_n\equiv0$ we define $\tilde{w}_n$ as the solution to \eqref{nlspow} with initial data $\tilde{w}_n(0)=P_n\varphi(x-x_n)$. If instead $t_n\to \infty$ (resp. $t_n\to-\infty$) we define $\tilde{w}_n$ as the solution to \eqref{nlspow} that scatters to $P_n\varphi(x-x_n)$ as $t\to\infty$ (resp. $t\to-\infty$). By persistence of regularity  for \eqref{nlspow} and Bernstein estimates, one has that
\begin{equation}\label{E:persist}
\|\nabla^s\tilde{w}_n\|_{L_t^\infty L_x^2}\lesssim \|\nabla^s\tilde{w}_n(0)\|_{L_x^2}\lesssim |x|^{(s-1)\theta}\|\varphi\|_{H^1}.
\end{equation}

We now define our approximate solutions 
\[
\tilde{v}_{n,T}(t,x)=\phi_n(x-x_n)\tilde{w}_n(t,x)\qtq{for}|t|<T,
\]
and for $|t|>T$ we define $\tilde{v}_{n,T}$ via the free Schr\"odinger evolution as above. Establishing \eqref{E:vn-intial-data} in the case that $t_n\equiv0$ is identical to the case above. In the case that $t_n\to\infty$ we estimate via the triangle inequality: for $t_n>T$,
\begin{align*}
\|\tilde{v}_{n,T}&(t_n)-e^{it_n\Delta}\varphi(x-x_n)\|_{H_x^1}\\
&\lesssim \|(\phi_n-1)\tilde{w}_{n}(T,\cdot+x_n)\|_{H_x^1}+\|\tilde{w}_{n}(T,\cdot+x_n)-e^{iT\Delta}\varphi\|_{H_x^1}.
\end{align*}
The first term above vanishes as $n\to\infty$ by the dominated convergence theorem, while the second vanishes due to the fact that $\tilde{w}_{n}(\cdot+x_n)$ scatters to $\varphi$ in $H^1$.

The scattering results of \cite{N,CHVZ} discussed in the introduction yield \eqref{E:vn-good-bounds}, so that it only remains to prove \eqref{E:vn-error-bounds}. 

The proof for $|t|>T$ follows as above, as $\tilde v_{n,T}$ are linear Schr\"odinger solutions, so that the proof follows as in the case $a_\pm=0$ above.

For $|t|<T$, we instead have
\begin{align}
e_{n,T}&=\Delta \phi_n(x-x_n)\,\tilde{w}_{n}+\nabla\phi_n\cdot\nabla\tilde{w}_{n}\label{E:vn-error5}\\
&\quad+a_+|\tilde{w}_{n}|^p\tilde{v}_{n,T}-a(x)|\tilde{v}_{n,T}|^p\tilde{v}_{n,T}\label{E:vn-error6}.
\end{align}

We estimate each of these terms separately on the space-time region $[-T,T]\times\R$. Starting with the first term in \eqref{E:vn-error5}, we have by H\"older's inequality and the symbol bounds on the $\phi_n$ that
\begin{align*}
\|\nabla[\Delta\phi_n(x-x_n)\tilde{w}_n]\|_{L_t^\frac{4}{3}L_x^1}
&\lesssim |T|^\frac{3}{4}\|\nabla\Delta\phi_n\|_{L_x^2}\|\tilde{w}_{n,T}\|_{L_t^\infty L_x^2}\\
&\quad+|T|^\frac{3}{4}\|\Delta\phi_n\|_{L_x^2}\|\nabla\tilde{w}_{n,T}\|_{L_t^\infty L_x^2}\\
&\lesssim |T|^\frac{3}{4}(x_n^{-\frac{5}{2}}+x_n^{-\frac{3}{2}})\|\varphi\|_{H_x^1}\to 0\qtq{as}n\to\infty.
\end{align*}
For the second term in \eqref{E:vn-error5} we estimate again via H\"older's inequality and \eqref{E:persist}:
\begin{align*}
\|\nabla[\nabla\phi_n\nabla\tilde{w}_n]\|_{L_t^\frac{4}{3}L_x^1}&\lesssim |T|^\frac{3}{4}\|\Delta\phi_n\|_{L_x^2}\|\nabla\tilde{w}_n\|_{L_t^\infty L_x^2}+|T|^\frac{3}{4}\|\nabla \phi_n\|_{L_x^2}\|\Delta\tilde{w}_n\|_{L_t^\infty L_x^2}\\
&\lesssim |T|^\frac{3}{4}(x_n^{-\frac{3}{2}}+x_n^{-\frac{1}{2}+\theta})\|\varphi\|_{H_x^1} \to 0 \qtq{as}n\to\infty.
\end{align*}
Finally, for \eqref{E:vn-error6} we estimate via the standard nonlinear estimate Lemma~\ref{L:StdNonlinEst}
\begin{align*}
\|\nabla&[a_+|\tilde{w}_{n}|^p\tilde{v}_{n,T}-a(x)|\tilde{v}_{n,T}|^p\tilde{v}_{n,T}]\|_{L_t^\frac{4}{3}L_x^1}\\
&\lesssim\|\nabla a\|_{L_x^\infty(x>\frac{1}{4}x_n)}\|\tilde{w}_{n}\|_{L_t^4L_x^\infty}\|\tilde{w}_{n}\|_{L_t^{2p}L_x^r}^p\\
&\quad+\|a_+-a(x)\|_{L_x^\infty(x>\frac{1}{4}x_n)}\|\tilde{w}_n\|_{L_t^4L_x^\infty}\|\tilde{w}_n\|_{L_t^{2p}L_x^r}^p.
\end{align*}
Using \eqref{E:vn-good-bounds} and the act that $a(x)\to a_+$ as $x\to\infty$, we see that this quantity tends to zero.  As the approximation by $C_c^\infty$-functions follows as in the case $a_+=0$, we finally complete the proof of Proposition~\ref{P:SFAP}. \end{proof}

%%%%%%%%%%%%%%%%%%%%%%%%%%%%%%%%%%%%%%%%%%%%%%%%

We turn to the proof of Proposition~\ref{P:PS}. 

\begin{proof}[Proof of Proposition~\ref{P:PS}] Let $\{u_n\}$ and $\{t_n\}$ be as in the statement of Proposition~\ref{P:PS}.  Applying a time translation, we may assume $t_n\equiv 0$.  Passing to a subsequence, we apply the linear profile decomposition (Proposition~\ref{P:LPD}) to $\{u_n(0)\}$ and write
\[
u_n(0,x) = \sum_{j=1}^J e^{it_n^j\Delta}\varphi^j(x-x_n^j) + w_n^J(x)
\]
for any $1\leq J\leq J^*$, with all of the properties guaranteed by Proposition~\ref{P:LPD}. To prove Proposition~\ref{P:PS}, we need to show that there is exactly one profile $\varphi^1$, with $t_n^1\equiv 0$, $x_n^1 \equiv 0$, and $w_n^1\to 0$ in $H^1$. 

Throughout the proof, we will make use of the notation
\[
\varphi_n^j(x) = e^{it_n^j\Delta}\varphi^j(x-x_n^j). 
\]

We first note that $J^*=0$ is not possible, for in this case we can apply the stability theory with the approximate solutions $\tilde u_n = e^{it\Delta}u_n(0)$ to derive uniform bounds on the $L_t^{2p} L_x^r$-norms of the $u_n$, contradicting the assumptions of Proposition~\ref{P:PS}.
% {\color{red}Let us first prove that at least one profile appears in the decomposition}, i.e. $J^*\geq 1$.  Suppose instead that $J^*=0$, so that
% \begin{equation}\label{PS:vanishing}
% \lim_{n\to\infty}\|e^{it\Delta}u_n(0)\|_{L_t^{2p}L_x^{r}(\R\times\R)}=0.
% \end{equation}
% Now define $\tilde u_n = e^{it\Delta} u_n(0)$ and observe $\tilde u_n(0)=u_n(0)$.  Using the standard nonlinear estimate and Strichartz estimates, one can verify from \eqref{PS:vanishing} that
% \[
% \| \nabla^s [(i\partial_t + \Delta)\tilde u_n - a|\tilde u_n|^p \tilde u_n]\|_{L_t^{\frac43} L_x^1}\to 0 \qtq{as}n\to\infty
% \]
% for $s\in\{0,1\}$. Thus by stability (Proposition~\ref{P:stability}) we can obtain that 
% \[
% \|u_n\|_{L_t^{2p}L_x^r}\lesssim 1
% \]
% uniformly for all large $n$, contradicting the assumptions of Proposition~\ref{P:PS}. 

We next rule out the case of multiple profiles (i.e. $J^*>1$).  To this end, we will prove 
\begin{equation}\label{one-profile}
\sup_j\limsup_{n\to\infty}M(\varphi^j)+E(\varphi_n^j)=E_c,
\end{equation}
which (using mass and energy decoupling) guarantees that $J^*=1$, as well as $w_n^1 \to 0$ in $H^1$.

Suppose instead that 
\[
\sup_j\limsup_{n\to\infty}M(\varphi^j)+E(\varphi_n^j)\leq E_c-3\delta
\]
for some $\delta>0$.  In this case, we will also derive uniform bounds on the scattering size of the $u_n$ for large $n$, once again contradicting the assumptions of Proposition~\ref{P:PS}.

The first step is to construct nonlinear solutions corresponding to each profile: 

If $(t_n^j,x_n^j)\equiv0$, we define $v^j$ to be the solution to \eqref{nls} with $v^j|_{t=0}=\varphi^j$. In this case,
\[
M(v^j)+E(v^j)=M(\varphi^j)+E(\varphi^j) \leq E_c-2\delta
\]
so that $v^j$ has finite global $L_t^{2p}L_x^r$-norm (and scatters in both time directions).

If instead $x_n^j\equiv 0$ and $t_n^j\to\pm\infty$, we define $v^j$ to be the solution to \eqref{nls} that scatters to $\varphi^j$ in $H^1$ as $t\to\pm\infty$.  In this case, we can also argue that
\[
M(v^j)+E(v^j) < E_c-\delta,
\]
so that $v^j$ has finite global $L_t^{2p} L_x^r$-norm  in this case as well.  In fact, scattering to $\varphi^j$ forces $M(v^j)=M(\varphi^j)$, so that it suffices to show that $E(v^j)\leq E(\varphi_n^j)+\delta$ for $n$ large.  To this end, we first note that by the dispersive estimate,
\[
E(e^{it\Delta}f) = \tfrac12 \|\nabla f\|_{L^2}^2 + o(1) \qtq{as} t\to\pm\infty 
\]
for $f\in H^1$.  Thus, by conservation of energy, $H^1$ scattering, and the assumption $t_n^j\to\pm\infty$, we have
\begin{align*}
E(v^j)  = E(e^{it\Delta}\varphi^j) + o(1)  &= \tfrac12\|\nabla\varphi\|_{L_x^2}^2 + o(1)   \\
&= \tfrac12\|\nabla\varphi_n^j\|_{L_x^2}^2 + o(1) = E(\varphi_n^j) + o(1)
\end{align*}
as $t\to\pm\infty$. This implies the desired bound. 

In either of these first two cases, we define
\[
v_n^j(t,x) = v^j(t+t_n^j,x),
\]
which yields a scattering solution to \eqref{nls}. 

Finally, if $|x_n^j|\to\infty$, we appeal to Proposition~\ref{P:SFAP} to construct the scattering nonlinear solution $v_n^j$ satisfying $v_n^j(0,x)=\varphi_n^j(x)$. 

We next define approximate solutions to \eqref{nls} via 
\[
u_n^J = \sum_{j=1}^J v_n^j + e^{it\Delta}w_n^J.
\]
By construction, we have that
\begin{equation}\label{PS-data}
\lim_{n\to\infty} \|u_n^J(0)-u_n(0)\|_{H^1} = 0 
\end{equation}
each $J$.  In addition, we will prove the following:
\begin{equation}\label{unJ-good-bounds}
\limsup_{J\to J^*}\limsup_{n\to\infty} [\|u_n^J \|_{L_t^{2p} L_x^r(\R\times\R)}+\|u_n^J\|_{L_t^\alpha L_x^\beta}+\|\nabla u_n^J\|_{L_t^\alpha L_x^\beta}]\lesssim 1
\end{equation}
for admissible $(\alpha,\beta)\not\in\{(\infty,2),(4,\infty)\}$ and
\begin{equation}\label{unJ-error-bounds}
\limsup_{J\to J^*}\limsup_{n\to\infty} \| \nabla^s[(i\partial_t + \Delta) u_n^J + a(x)|u_n^J|^p u_n^J]\|_{L_t^{\frac43} L_x^1} = 0
\end{equation}
for $s\in\{0,1\}$. By stability (Proposition~\ref{P:stability}), \eqref{PS-data}--\eqref{unJ-error-bounds} imply global $L_t^{2p} L_x^r$-bounds for the $u_n$ for all large $n$, contradicting the assumptions of Proposition~\ref{P:PS} (and hence proving \eqref{one-profile}). 

Before proving \eqref{unJ-good-bounds} and \eqref{unJ-error-bounds}, we first discuss our strategy and introduce some notation. The key to establishing these estimates to is combine the good bounds obeyed by the solutions $v_n^j$ together with the orthogonality of the profiles (a consequence of the profile decomposition) and the vanishing of the remainders $w_n^J$. To utilize the orthogonality, we will need to approximate by compactly supported functions of spacetime. Therefore, we must avoid endpoint Strichartz norms in the estimates, which necessitates some interpolation in what follows (as our usual nonlinear estimate makes use of the $L_t^4 L_x^\infty$-norm). As long as we restrict to powers $p>2$, however, there is always room to interpolate.

We define
\[
Q(p):=\begin{cases}
p & p>4\\
\frac{8p}{p+2} & 2< p\leq 4.
\end{cases}
\]
The relevance of this exponent is that the space $L_t^{Q(p)}L_x^\infty$ has the scaling of $\dot{H}^{s(p)}$ (cf. \eqref{sp}). In particular, if $Q>Q(p)$ and $R$ is such that $\frac{2}{Q}+\frac{1}{R}=\frac{1}{2}-s(p)$, then $R<\infty$ and by Sobolev embedding we have
\begin{equation}\label{E:InterpIdea}
\|f\|_{L_t^QL_x^R}\lesssim \||\nabla|^{s(p)}f\|_{L_t^QL_x^\frac{2Q}{Q-4}}\lesssim \|f\|_{L_t^QL_x^\frac{2Q}{Q-4}}+\|\nabla f\|_{L_t^QL_x^\frac{2Q}{Q-4}}.
\end{equation}
Thus if we control $\|(\nabla)^sf\|_{Q,\frac{2Q}{Q-4}}$ for $s\in\{0,1\}$, we also control $\|f\|_{L_t^QL_x^R}$.

We now record the following modified nonlinear estimate, which follows immediately from H\"older's inequality.
\begin{lemma}\label{L:InterpEstV2}
Let $\theta\in (0,1)$ satisfy
\[
\theta>\begin{cases}
\frac{2}{p}-1 & p>4\\
\frac{5p-2-p^2}{2p}& 2<p\leq 4
\end{cases}\qtq{and define}(Q,R)=(\tfrac{4p(p-1)}{(3-\theta)p-2},\tfrac{2r(p-1)}{2(p-1)-r(1-\theta)}).
\]
Then
\begin{align*}
\|afgh^{p-1}\|_{L_t^\frac{4}{3}L_x^1} & \lesssim \|a\|_{L_x^\rho}\|fg\|_{L_t^\frac{4p}{\theta p+2}L_x^\frac{2r}{r(1-\theta)+2}}\|h\|^{p-1}_{L_t^Q L_x^R} \\
&\lesssim \|a\|_{L_x^\rho}\|f\|_{L_t^\frac{4}{\theta}L_x^\frac{2}{1-\theta}}\|g\|_{L_t^{2p}L_x^r}\|h\|_{L_t^Q L_x^R}^{p-1}.
\end{align*}
The pair $(\frac{4}{\theta},\frac{2}{1-\theta})$ is admissible and the pair $(Q,R)$ satisfies 
\[
Q(p)<Q<2p\qtq{and}\tfrac{2}{Q}+\tfrac{1}{R}=\tfrac{1}{2}-s(p).
\]
\end{lemma}

We now state the orthogonality conditions we will need.

\begin{lemma}[Orthogonality]\label{L:orthogonality} For $j\neq k$, $\theta$ as above, $s\in\{0,1\}$, and for admissible pairs $(\alpha,\beta)\not\in\{(\infty,2),(4,\infty)\}$ we have
\[
\lim_{n\to\infty} \|(\nabla)^s v_n^j\, v_n^k\|_{L_t^{\frac{4p}{p\theta+2}}L_x^{\frac{2r}{r(1-\theta)+2}}}+\|v_n^jv_n^k\|_{L_t^pL_x^\frac{r}{2}}+\|(\nabla)^sv_n^jv_n^k\|_{L_t^\frac{\alpha}{2}L_x^\frac{\beta}{2}}=0.
\]
\end{lemma}

\begin{proof}[Proof of Lemma~\ref{L:orthogonality}] We will only show the case of the $L_t^\frac{4p}{p\theta+2}L_x^\frac{2r}{r(1-\theta)+2}$-norm with $s=1$, as the others are similar. We use the notation $(T_n^jf)(t,x)=f(t+t_n^j,x-x_n^j)$. 

First, let $\eps>0$ choose $\psi^j$ and $\psi^k$ so that for all $n$ sufficiently large,
\[
\|\nabla(v_n^j-T_n^j\psi^j)\|_{L_t^\frac{4}{\theta}L_x^\frac{2}{1-\theta}}+\|v_n^k-T_n^k\psi^k\|_{L_t^{2p}L_x^r}<\varepsilon.
\]
If $x_n^j\equiv0$ this can be done trivially as $v_n^j$ is obtained from $v^j$ by applying translations, and likewise for $x_n^k$. Otherwise, if $|x_n^j|\to\infty$ this can be done by appealing to Proposition~\ref{P:SFAP}.  Using the triangle inequality, H\"older's inequality, the bounds on $v_n^j$ and $v_n^k$, and the asymptotic orthogonality of the parameters, one can now readily deduce
\[
\|\nabla v_n^j v_n^k\|_{L_t^\frac{4p}{p\theta+2}L_x^\frac{2r}{r(1-\theta)+2}}\lesssim_{E_c} \eps + o_n(1)\qtq{as}n\to\infty,
\]
which yields the result. \end{proof}

% Now we observe from the triangle and H\"older's inequality that
% \begin{align*}
% \|\nabla& v_n^j v_n^k\|_{L_t^\frac{4p}{p\theta+2}L_x^\frac{2r}{r(1-\theta)+2}}\\
% &\lesssim\|\nabla v_n^j(v_n^k-T_n^k \psi^k)\|_{L_t^\frac{4p}{p\theta+2}L_x^\frac{2r}{r(1-\theta)+2}}+\|\nabla(v_n^j-T_n^j\psi^j)T_n^k\psi^k\|_{L_t^\frac{4p}{p\theta+2}L_x^\frac{2r}{r(1-\theta)+2}}\\
% &\quad+\|\nabla T_n^j\psi^j\cdot T_n^k\psi^k\|_{L_t^\frac{4p}{p\theta+2}L_x^\frac{2r}{r(1-\theta)+2}}\\
% &\lesssim \|\nabla v_n^j\|_{L_t^\frac{4}{\theta}L_x^\frac{2}{1-\theta}}\|v_n^k-T_n^k\psi^k\|_{L_t^{2p}L_x^r}+\|\psi^k\|_{L_t^{2p}L_x^r}\|\nabla(v_n^j-T_n^j\psi^j)\|_{L_t^\frac{4}{\theta}L_x^\frac{2}{1-\theta}}\\
% &\quad+\|\nabla T_n^j\psi^j\cdot T_n^k\psi^k\|_{L_t^\frac{4p}{p\theta+2}L_x^\frac{2r}{r(1-\theta)+2}}\\
% &\lesssim_{E_c}\varepsilon+o_n(1).
% \end{align*}
% Since $\varepsilon$ was arbitrary the result follows.

We now turn to the proof of \eqref{unJ-good-bounds}. We will show the $L_t^{2p}L_x^r$-bound, as the others follow similarly. Let $0<\eta\ll1$. By $H^1$ decoupling, we may find some $J_0=J_0(\eta)$ so that
\begin{equation}\label{E:SmallDataEta}
\sum_{j>J_0}\|\varphi^j\|_{H^1}^2<\eta.
\end{equation}
Thus, by the small data theory, we have that 
\[
\sum_{j>J_0}\|v_n^j\|_{L_t^{2p}L_x^r}^2\lesssim \eta.
\]

By the triangle inequality, \eqref{unJ-good-bounds} will follow if we can show
\[
\limsup_{J\to J^*}\limsup_{n\to\infty}\,\biggl\|\sum_{j=1}^Jv_n^j\biggr\|_{L_t^{2p} L_x^r}\lesssim 1.
\]
Indeed, for any finite $J>J_0$ we can estimate using the triangle inequality and Lemma~\ref{L:orthogonality} to obtain
\begin{align*}
\biggl\|\sum_{j=1}^Jv_n^j\biggr\|_{L_t^{2p}L_x^r}^2&\lesssim \sum_{j=1}^{J_0}\|v_n^j\|_{L_t^{2p}L_x^r}^2+\sum_{j>J_0}\|v_n^j\|_{L_t^{2p}L_x^r}^2+\sum_{j\not=k}\|v_n^jv_n^k\|_{L_t^pL_x^\frac{r}{2}}\\
&\lesssim \sum_{j=1}^{J_0}\|v_n^j\|_{L_t^{2p}L_x^r}^2+\eta+o_n(1),
\end{align*}
so that \eqref{unJ-good-bounds} follows.

We now turn to the proof of \eqref{unJ-error-bounds}.  We denote
\[
F(x,z) = a(x)f(z) = a(x) |z|^p z. 
\]
We then have
\begin{align*}
(i\partial_t+ \Delta) u_n^J - F(x,u_n^J) & = \sum_{j=1}^J F(x,v_n^j) - F\bigl(x,\sum_{j=1}^J v_n^j\bigr) \\
& \quad + F(x,u_n^J - e^{it\Delta} w_n^J) - F(x,u_n^J). 
\end{align*}
We will show that for $s\in\{0,1\}$,
\begin{equation}\label{Jerror1}
\limsup_{J\to J^*}\limsup_{n\to\infty} \biggl\|(\nabla)^s\biggl[\sum_{j=1}^J F(x,v_n^j) - F\bigl(x,\sum_{j=1}^J v_n^j\bigr)\biggr]\biggr\|_{L_t^{\frac43}L_x^1(\R\times\R)}=0
\end{equation}
and
\begin{equation}\label{Jerror2}
\limsup_{J\to J^*}\limsup_{n\to\infty} \| (\nabla)^s[F(x,u_n^J-e^{it\Delta}w_n^J) - F(x,u_n^J)\|_{L_t^{\frac43}L_x^1(\R\times\R)}=0.
\end{equation}

\begin{proof}[Proof of \eqref{Jerror1}] The estimate given for $s=1$ will also show how to treat the simpler case $s=0$, so without loss of generality we fix $s=1$.  We begin by applying the product rule to estimate
\begin{align}
\biggl|\nabla\biggl[\sum_{j=1}^J F(x,v_n^j) - F(x,\sum_{j=1}^J v_n^j)\biggr]\biggr| &\lesssim |\nabla a|\,\biggl|\sum_{j=1}^J f(v_n^j) - f\biggl(\sum_{j=1}^J v_n^j\biggr)\biggr| \label{Jerror11} \\
& \quad + |a|\,\biggl| \sum_{j=1}^J \nabla [f(v_n^j)] - \nabla \biggl[f\biggl(\sum_{J=1}^J v_n^j\biggr)\biggr]\biggr|. \label{Jerror12}
\end{align}

The term \eqref{Jerror12} is more involved (and will demonstrate how to handle \eqref{Jerror11} as well), so we focus on this term first.  We begin by applying the pointwise estimate Lemma~\ref{L:PW2} to write
\[
|\eqref{Jerror12}| \lesssim_J \sum_{j\neq k} |\nabla v_n^j| |v_n^k|^p + \sum_{j\neq k} |\nabla v_n^j| |v_n^j|^{p-1} |v_n^k|.
\]
It therefore suffices to prove that for any $j\neq k$ and any $\ell$, we have
\[
\lim_{n\to\infty} \| a\,|\nabla v_n^j| |v_n^k| |v_n^\ell|^{p-1}\|_{L_t^{\frac43} L_x^1} = 0. 
\]
To this end, we choose the parameter $\theta$ as in Lemma~\ref{L:InterpEstV2} and use Lemma~\ref{L:orthogonality} to obtain
\begin{align*}
\| a\,|\nabla v_n^j| |v_n^k| |v_n^\ell|^{p-1}\|_{L_t^{\frac43} L_x^1} &\lesssim\|a\|_{L_x^\rho}\|\nabla v_n^jv_n^k\|_{L_t^\frac{4p}{p\theta+2}L_x^\frac{2r}{r(1-\theta)+2}}\|v_n^\ell\|^{p-1}_{L_t^Q L_x^R}\\
&\to 0\qtq{as}n\to\infty.
\end{align*}
For \eqref{Jerror11}, we rely directly on pointwise estimates reduce matters to proving
\[
\|\nabla a\ v_n^j |v_n^k|^p\|_{L_t^{\frac43} L_x^1} \to 0\qtq{as}n\to\infty
\]
for any $j\neq k$. In this case we estimate this term by
\[
\|\nabla a\|_{L_x^\rho}\|v_n^jv_n^k\|_{L_t^\frac{4p}{p\theta+2}L_x^\frac{2r}{r(1-\theta)+2}}\|v_n^k\|_{L_t^Q L_x^R}^{p-1}\to 0\qtq{as}n\to\infty,
\]
where we have again used Lemma~\ref{L:orthogonality}. \end{proof}

\begin{proof}[Proof of \eqref{Jerror2}] The estimate given for $s=1$ will also show how to treat the simpler case $s=0$, so without loss of generality we fix $s=1$. We begin by applying the product rule and the pointwise estimate Lemma~\ref{L:PW} to obtain
\begin{align}
\|&\nabla  [ F(x,u_n^J-e^{it\Delta}w_n^J)-F(x,u_n^J)]\|_{L_t^{\frac43} L_x^1} \nonumber \\
& \lesssim \|\nabla a\,[f(u_n^J-e^{it\Delta}w_n^J)-f(u_n^J)]\|_{L_t^{\frac43}L_x^1}  + \| a\, \nabla[f(u_n^J - e^{it\Delta}w_n^J)-f(u_n^J)]\, \|_{L_t^{\frac43} L_x^1}\nonumber  \\
& \lesssim \|\nabla a[f(u_n^J-e^{it\Delta}w_n^J)-f(u_n^J)]\|_{L_t^{\frac43}L_x^1} \label{Jerror21}\\
& \quad + \| a\, |\nabla e^{it\Delta}w_n^J|\, | u_n^J|^p \|_{L_t^{\frac43} L_x^1} \label{Jerror22}\\
& \quad + \| a\, |\nabla e^{it\Delta} w_n^J|\, |e^{it\Delta} w_n^J|^p \|_{L_t^{\frac43} L_x^1} \label{Jerror23}\\
& \quad + \| a\, |\nabla u_n^J |\, |e^{it\Delta} w_n^J|^p \|_{L_t^{\frac43} L_x^1} \label{Jerror24}\\
& \quad + \| a\, |\nabla u_n^J|\, |e^{it\Delta} w_n^J|\, |u_n^J|^{p-1} \|_{L_t^\frac43 L_x^1}. \label{Jerror25}
\end{align}

Let first handle the most difficult term, namely \eqref{Jerror22}. The difficulty stems from the fact that $e^{it\Delta} w_n^J$ is only assumed to vanish in a space without derivatives.  The remedy will be to apply the local smoothing estimate (Proposition~\ref{P:LS}).

To get started with this term, we first estimate
\[
\eqref{Jerror22}\lesssim \Biggl\|a\,|\nabla e^{it\Delta}w_n^J|\,\Biggl|\sum_{j=1}^Jv_n^j\Biggr|^p\Biggr\|_{L_t^\frac43L_x^1}+\|a\,|\nabla e^{it\Delta}w_n^J|\,|e^{it\Delta}w_n^J|^p\|_{L_t^\frac43L_x^1}.
\]
The latter summand above can be combined with \eqref{Jerror23}, so we are left with estimating the first. Let $\eta>0$ be given and choose $J_0$ as in \eqref{E:SmallDataEta}. We then write
\[
\Biggl|\sum_{j=1}^Jv_n^j\Biggr|^p=\sum_{j=1}^{J_0}|v_n^j|^p+\sum_{j\geq J_0}|v_n^j|^p+\mathcal{O}_J\biggl(\sum_{j\not=k}|v_n^j|^{p-1}|v_n^k|\biggr)
\]
The contribution of the $j\geq J_0$ sum is controlled by
\begin{align*}
\bigl\|a\,|\nabla e^{it\Delta} w_n^J|\sum_{j\geq J_0}|v_n^j|^p\bigr\|_{L_t^\frac{4}{3}L_x^1}&\lesssim \|a\|_{L_x^\rho}\|\nabla e^{it\Delta}w_n^J\|_{L_t^4L_x^\infty}\sum_{j\geq J_0}\|v_n^j\|_{L_t^{2p}L_x^r}\\
&\lesssim \|w_n^J\|_{H^1}\sum_{j\geq J_0}\|\varphi^j\|_{H^1}^p\lesssim \eta,
\end{align*}
for all large $n$, uniformly in $J$. By the standard nonlinear estimate Lemma~\ref{L:StdNonlinEst} and orthogonality (Lemma~\ref{L:orthogonality}) the contribution of the $j\not=k$ sum is controlled by
\begin{align*}
\bigl\|a&\,|\nabla e^{it\Delta}w_n^J|\mathcal{O}\bigl(\sum_{j\not=k}|v_n^j|^{p-1}|v_n^k|\bigr)\|_{L_t^\frac{4}{3}L_x^1}\\
&\lesssim_J \|a\|_{L_x^\rho}\|\nabla e^{it\Delta}w_n^J\|_{L_t^4L_x^\infty}\sum_{j\not=k}\|v_n^jv_n^k\|_{L_t^pL_x^\frac{r}{2}}\|v_n^j\|_{L_t^{2p}L_x^r}^{p-2}\\
&\lesssim_Jo_n(1)\qtq{as}n\to\infty.
\end{align*}
For the remaining term, we apply the triangle inequality and find that it suffices to prove that for fixed $j\leq J_0$, we have
\[
\lim_{J\to\infty}\lim_{n\to\infty}\|a\,|\nabla e^{it\Delta}w_n^J||v_n^j|^p\|_{L_t^\frac{4}{3}L_x^1}=0.
\]
Now, recalling the small parameter $\eta>0$ introduced above, we may find $\psi_\eta\in C_c^\infty(\R\times\R)$ such that
\[
\|v_n^j-\psi_\eta(\cdot+t_n^j,\cdot-x_n^j)\|_{L_t^{2p}L_x^r}<\eta,
\]
as in the proof of Lemma~\ref{L:orthogonality}. Applying H\"older's inequality and the standard nonlinear estimate, we therefore find that (after a change of variables) that it suffices to prove
\[
\lim_{J\to\infty}\lim_{n\to\infty}\|[\tau_n^j a]|\nabla e^{i(t-t_n^j)\Delta}\tau_n^jw_n^J|\,|\psi_\eta|^p\|_{L_t^\frac{4}{3}L_x^1}=0,
\]
where $\tau_n^jg(x):=g(x+x_n^j)$. In fact, writing $K$ for the support of $\psi_\eta$, we have by H\"older's inequality, the local smoothing estimate Proposition~\ref{P:LS}, and a change of variables that
\begin{align*}
[\tau_n^ja]|\nabla e^{i(t-t_n^j)\Delta}\tau_n^jw_n^J|\,|\psi_\eta|^p\|_{L_t^\frac{4}{3}L_x^1}&\lesssim_K \|a\|_{L^\infty}\|\nabla e^{i(t-t_n^j)\Delta}\tau_n^jw_n^J\|_{L_{t,x}^2(K)}\\
&\lesssim_K\|e^{it\Delta}w_n^J\|_{L_t^{2p}L_x^r}^\frac{1}{3}\|\nabla w_n^J\|_{L_t^2}^\frac{2}{3}\\
&=o(1)\qtq{as}n,J\to\infty.
\end{align*}
This completes the treatment of \eqref{Jerror22}.

The estimates of \eqref{Jerror21} and \eqref{Jerror23} are simpler:

For \eqref{Jerror21}, we use \eqref{E:BPW1} and are left to estimate
\begin{align}
\|\nabla a[f(u_n^J-e^{it\Delta}w_n^J)-f(u_n^J)]\|_{L_t^{\frac43}L_x^1} & \lesssim \|\nabla a\|_{L^\rho} \|e^{it\Delta} w_n^J \|_{L_t^{2p} L_x^{r}}^p \|e^{it\Delta}w_n^J\|_{L_t^4 L_x^\infty} \label{Jerror211}\\
& \quad + \|\nabla a\ |u_n^J|^p\ e^{it\Delta} w_n^J\|_{L_t^{\frac43} L_x^1}. \label{Jerror212}
\end{align}
The error term \eqref{Jerror211} vanishes as $n,J\to\infty$. To treat \eqref{Jerror212} is not difficult, but it requires interpolation as above, as we need to avoid the endpoint Strichartz norms for $u_n^J$. Applying Lemma~\ref{L:InterpEstV2} we observe
\begin{align*}
\eqref{Jerror212}\lesssim \|\nabla a\|_{L^\rho}\|u_n^J\|_{L_t^\frac{4}{\theta}L_x^\frac{2}{1-\theta}}\|e^{it\Delta}w_n^J\|_{L_t^{2p}L_x^r}\|u_n^J\|_{L_t^Q L_x^R}^{p-1},
\end{align*}
which tends to zero as $n, J\to\infty$ (cf. Lemma~\ref{L:InterpEstV2} above).

We turn to \eqref{Jerror23}, by the usual nonlinear estimate Lemma~\ref{L:StdNonlinEst}, 
\begin{align*}
\| a\, |\nabla e^{it\Delta} w_n^J|\, |e^{it\Delta} w_n^J|^p \|_{L_t^{\frac43} L_x^1} & \lesssim \|a\|_{L^\rho} \| \nabla e^{it\Delta} w_n^J\|_{L_t^4 L_x^\infty} \|e^{it\Delta} w_n^J\|_{L_t^{2p} L_x^r}^p \\
& \lesssim \|w_n^J\|_{H^1} \|e^{it\Delta}w_n^J\|_{L_t^{2p} L_x^r}^p \\
& = o(1) \qtq{as} n,J\to\infty. 
\end{align*}

For \eqref{Jerror24}, we must again interpolate and so we apply Lemma~\ref{L:InterpEstV2}:
\begin{align*}
\|a&|\nabla u_n^J||e^{it\Delta}w_n^J|^p\|_{L_t^\frac{4}{3}L_x^1}\\
&\lesssim \|a\|_{L_x^\rho}\|\nabla u_n^J\|_{L_t^\frac{4}{\theta}L_x^\frac{2}{1-\theta}}\|e^{it\Delta}w_n^J\|_{L_t^{2p}L_x^r}\|e^{it\Delta}w_n^J\|_{L_t^Q L_x^R}^{p-1} \to 0 \qtq{as}n,J\to\infty.
\end{align*}

Finally, we turn to \eqref{Jerror25}, which is treated similarly:
\begin{align*}
\|a&\,|\nabla u_n^J||e^{it\Delta}w_n^J||u_n^J|^{p-1}\|_{L_t^\frac{4}{3}L_x^1}\\
&\lesssim\|a\|_{L_x^\rho}\|\nabla u_n^J\|_{L_t^\frac{4}{\theta}L_x^\frac{2}{1-\theta}}\|e^{it\Delta}w_{n}^J\|_{L_t^{2p}L_x^r}\|u_n^J\|_{L_t^Q L_x^R}^{p-1} \to 0 \qtq{as}n,J\to\infty.
\end{align*}
This completes the proof of \ref{Jerror2}.  \end{proof}

With \eqref{PS-data}--\eqref{unJ-error-bounds} established we complete the proof of \eqref{one-profile}.  In particular, $J^*=1$ and we have the decomposition
\[
u_n(0,x)=e^{it_n^1\Delta}\varphi^1(x-x_n^1)+w_n^1(x)
\]
with $w_n^1\to 0$ in $H^1$. To finish the proof we would like to show that $(t_n^1,x_n^1)\equiv0$.

If $|x_n^1|\to\infty$, then an application of Proposition~\ref{P:SFAP} and Proposition~\ref{P:stability} implies spacetime bounds for $u_n$, leading to a contradiction.  

If $t_n^1\to\infty$ we set
\[
\tilde{u}_n(t)=e^{it\Delta}u_n(0)=e^{i(t+t_n^1)\Delta}\varphi^1_n+e^{it\Delta}w_n^1,
\]
which satisfies $\tilde u_n(0)=u_n(0)$, satisfies global space-time bounds, and satisfies
\begin{align*}
\|(\nabla)^s&[(i\partial_t+\Delta)\tilde{u}_n-a|\tilde{u}_n|^p\tilde{u}_n]\|_{L_t^\frac{4}{3}L_x^1([0,\infty)\times\R)}\to 0
\qtq{as}n\to\infty
\end{align*}
for $s\in\{0,1\}$. Thus by stability, we can deduce spacetime bounds for $u_n$ on $[0,\infty)$, a contradiction.  Similarly, if $t_n^1\to-\infty$, we can deduce bounds on $(-\infty,0]$.  Thus we obtain $(t_n^1,x_n^1)\equiv (0,0)$ and thereby complete the proof of Proposition~\ref{P:PS}. \end{proof}

%%%%%%%%%%%%%%%%%%%%%%%%%%%%%%%%%%%%%%%%%%%%%%%%%%%%%%%%%%%%%
\section{Preclusion of compact solutions}\label{S:Preclusion}

In this section we derive a Morawetz estimate in the spirit of Nakanishi \cite{N} and use it to preclude the possibility of compact solutions as in Corollary~\ref{C:AP}, thus completing the proof of Theorem~\ref{T}.

We introduce the notation
\[
\|f\|_{Z(t)} = \| c(t,x)(x+2it\nabla)f\|_{L_x^2},
\]
for $f\in H^1$, where
\[
c(t,x) = \biggl[\frac{t}{\langle t\rangle^3+|x|^3}\biggr]^{\frac12}.
\]
We note that $\|\cdot\|_{Z(t)}$ defines a norm for $t\geq 1$. To see that $\|f\|_{Z(t)}=0$ implies $f\equiv 0$, for example, we may use the identity 
\[
(x+2it\nabla)f = \exp\bigl\{\tfrac{i|x|^2}{4t}\bigr\} 2it\nabla\bigl[\exp\bigl\{-\tfrac{i|x|^2}{4t}\bigr\} f\bigr].
\]
We also observe that
\[
\|u(t)\|_{Z(t)} \lesssim \|u(t)\|_{H^1}
\]
uniformly in $t\geq 1$.  Indeed, this follows from the fact that
\[
\| x c(t,x)\|_{L_{t,x}^\infty} + \| t c(t,x)\|_{L_{t,x}^\infty} \lesssim 1.
\]

We turn to the proof of the Morawetz estimate, working in general dimension $d$. 
\begin{proposition}[Morawetz estimate]\label{P:Morawetz} Suppose $a:\R^d\to\R$ is nonnegative and repulsive.  Let $u:\R^d\times\R\to\C$ be an $H^1$ solution to \eqref{nls}.  Then
\[
\int_1^\infty\int_{\R^d}\frac{|(x+2it\nabla)u|^2}{\langle t\rangle^3+|x|^3}+\frac{\langle t\rangle^2a(x)|u(t,x)|^{p+2}}{\langle t\rangle^3+|x|^3}\,dx\,dt\lesssim_{M(u),E(u)} 1.
\]
In particular,
\[
\int_1^\infty \|u(t)\|_{Z(t)}^2\, \tfrac{dt}{t} \lesssim_{M(u),E(u)} 1.
\]
\end{proposition}

\begin{proof} Given a function $b=b(t,x)$ and a solution $u$ to \eqref{nls}, let us define the \emph{Morawetz quantity}
\[
M_b^0(t)=\tfrac{d}{dt}\int|u|^2b\,dx=\int 2\nabla b\cdot\text{Im}(\overline{u}\nabla u)+|u|^2b_t\,dx.
\]

We compute the time derivative of $M_b^0$, employing Einstein summation convention in what follows. Using \eqref{nls} and integrating by parts, 
\begin{align*}
\tfrac{d}{dt}M_b^0(t)&=\int 2\tfrac{d}{dt}\text{Im}(b_j\overline{u}u_j)+\tfrac{d}{dt}(|u|^2b_t)\,dx\\
&=2\text{Im}\int b_j[\overline{u}_tu_j+\overline{u}(u_t)_j]+2(b_t)_j\overline{u}u_j\,dx+\int|u|^2b_{tt}\,dx\\
&=2\text{Re}\int b_j[-\overline{u}_{kk}u_j+\overline{u}u_{jkk}-a_j|u|^{p+2}-\tfrac{pa}{p+2}\partial_j(|u|^{p+2})]\,dx\\
&\quad\quad\quad\quad+\int 4\text{Im}\bigl((b_t)_j\overline{u}u_j\bigr)+|u|^2b_{tt}\,dx\\
&=\int 4\text{Re}(\overline{u}_ju_kb_{jk})-|u|^2b_{jjkk}\,dx+\tfrac{2}{p+2}\int b_{jj}a|u|^{p+2}-\tfrac{4}{p+2}b_ja_j|u|^{p+2}\,dx\\
&\quad\quad+\int4\text{Im}\Bigl((b_t)_j\overline{u}u_j\Bigr)+|u|^2b_{tt}\,dx.
\end{align*}

We now specialize to $b(t,x)=\sqrt{t^2+|x|^2}$. Writing
\[
\slashed\nabla_0 f = \nabla f - \bigl[\tfrac{x}{|x|}\cdot\nabla f\bigr]\tfrac{x}{|x|},
\]
explicit computation leads to 
\begin{align*}
\tfrac{d}{dt}M_b^0(t)&=\int\frac{|(x+2it\nabla)u|^2}{(t^2+|x|^2)^{3/2}}+\frac{4|x|^2|\slashed\nabla_0u|^2}{(t^2+|x|^2)^{3/2}}\,dx+\int\frac{-\Delta\Delta b\,|u|^2}{(t^2+|x|^2)^{3/2}}\,dx\\
&\quad+\frac{2}{p+2}\int\frac{\bigl[(d-1)|x|^2+dt^2\bigr]a(x)|u|^{p+2}}{(t^2+|x|^2)^{3/2}}\,dx-\frac{4}{p+2}\int\frac{x\cdot\nabla a(x)|u|^{p+2}}{(t^2+|x|^2)^{3/2}}\,dx,
\end{align*}
where 
\[
|\Delta\Delta b|=\biggl|\frac{(d-1)(d-3)}{(t^2+|x|^2)^{3/2}}+\frac{6(d-3)t^2}{t^2+|x|^2)^{5/2}}+\frac{15t^4}{(t^2+|x|^2)^{7/2}}\biggl| \lesssim_d t^{-3}.
\]

We now use the Fundamental Theorem of Calculus to obtain
\begin{equation}\label{MorEqn}
\int_1^\infty\tfrac{d}{dt}M_b^0(t)\,dt \lesssim\sup_{t\in [1,\infty)}|M_b^0(t)|.
\end{equation}

By Hölder's inequality, we bound the right-hand side of \eqref{MorEqn} by
\begin{align*}
\sup |M_b^0(t)|& \lesssim \|\nabla b\|_{L_{t,x}^\infty}\|u\|_{L_t^\infty L_x^2}\|\nabla u\|_{L_t^\infty L_x^2}+\|b_t\|_{L_{t,x}^\infty}\|u\|_{L_t^\infty L_x^2}^2 \\
& \lesssim\|u\|_{L_t^\infty H_x^1}^2.
\end{align*}

For the left-hand side of \eqref{MorEqn}, we use the repulsive assumption $x\cdot\nabla a\leq 0$ to obtain the following \emph{lower} bound:
\begin{align*}
\int_1^\infty \tfrac{d}{dt} M_b^0(t)\,dt \geq \int_1^\infty \int_{\R^d} \frac{|(x+2it\nabla)u|^2}{(t^2+|x|^2)^{3/2}} + \frac{t^2 a|u|^{p+2}}{(t^2+|x|^2)^{3/2}} - \frac{|\Delta\Delta b|\,|u|^2}{(t^2+|x|^2)^{3/2}}\,dx\,dt.
\end{align*}
Using the bound on $\Delta\Delta b$, we obtain 
\begin{align*}
\int_1^\infty \int_{\R^d} & \frac{|(x+2it\nabla)u|^2}{(t^2+|x|^2)^{3/2}} + \frac{t^2 a|u|^{p+2}}{(t^2+|x|^2)^{3/2}}\,dx\,dt \\
& \lesssim \|u\|_{L_t^\infty H_x^1}^2 + \|u\|_{L_t^\infty L_x^2}^2\int_1^\infty \int_{\R^d} \frac{dx\,dt}{t^3(t^2+|x|^2)^{3/2}},
\end{align*}
which implies the desired estimate.\end{proof}

The final ingredient we need is the following \emph{lower} bound for compact solutions in the $Z(t)$-norm. 

\begin{proposition}\label{P:LB} If $u$ is a nonzero $H^1$ solution to \eqref{nls} such that $\{u(t):t\in \R\}$ is pre-compact in $H^1$, then
\[
\inf_{t\geq 1}\|u(t)\|_{Z(t)}\gtrsim 1.
\]
\end{proposition}

\begin{proof} Suppose $u$ is a nonzero $H^1$ solution to \eqref{nls} such that $\{u(t):t\in \R\}$ is pre-compact in $H^1$.  Suppose towards a contradiction that $\{t_n\}\subset[1,\infty)$ is such that
\[
\|u(t_n)\|_{Z(t_n)}\to 0. 
\]
Passing to a subsequence, we may assume that $u(t_n)\to\varphi$ strongly in $H^1$ for some nonzero $\varphi\in H^1$. Passing to a further subsequence, we may assume that $t_n\to t_0\in[1,\infty)$ or $t_n\to\infty$. If $t_n\to t_0$, then we must have $\|u(t_0)\|_{Z(t_0)}=0$, so that by continuity of the flow, we have $u(t_0)=\varphi=0$, a contradiction.  

Suppose instead that $t_n\to\infty$.  In this case, we first observe that
\begin{align*}
\|\varphi\|_{Z(t_n)} & \lesssim \|u(t_n)\|_{Z(t_n)} + \| u(t_n)-\varphi\|_{Z(t_n)} \\
& \lesssim \|u(t_n)\|_{Z(t_n)} + \|u(t_n)-\varphi\|_{H^1} \to 0 
\end{align*}
as $n\to\infty$. However, by the dominated convergence theorem, since $t_n\to\infty$ we obtain
\[
\|\varphi\|_{Z(t_n)} = \|2\nabla \varphi\|_{L^2}+ o(1)
\]
as $n\to\infty$.  Thus we again obtain the contradiction $\varphi=0$.\end{proof}

Combining Proposition~\ref{P:Morawetz} and Proposition~\ref{P:LB} immediately yields the following.

\begin{corollary}\label{C:contradiction} Let $a$ satisfy the hypotheses of Theorem~\ref{T}.  Then there are no nonzero $H^1$ solutions $u$ to \eqref{nls} such that $\{u(t):t\in\R\}$ is pre-compact in $H^1$.
\end{corollary}

Combining Corollary~\ref{C:AP} and Corollary~\ref{C:contradiction}, we complete the proof of Theorem~\ref{T}.

\appendix

\end{document}